\documentclass[11pt]{article}

\usepackage[a4paper,margin=1in]{geometry}  % 1-inch margins all around

\usepackage{amsmath}
\usepackage{amssymb}
\usepackage{bm}
\usepackage{graphicx}
\usepackage{booktabs}
\usepackage{multirow}
\usepackage{caption}
\usepackage{subcaption}
\usepackage{soul}
\usepackage{tikz}
\usepackage{authblk}   % optional, for nicer affiliations
\usepackage{hyperref}

\usetikzlibrary{arrows.meta, decorations.pathmorphing, calc, positioning}

\newcommand{\cred}[1]{{\color{red}  #1}}
\newcommand{\cblue}[1]{{\color{blue}  #1}}

\newtheorem{definition}{Definition}[section]
\newtheorem{theorem}{Theorem}[section]
\newtheorem{proposition}{Proposition}[section]
\newtheorem{remark}{Remark}[section]

\title{Analysis of eigenvalue clustering leads to optimal scaling in numerical radiative transfer}

\author[1,2]{Pietro Benedusi\thanks{Corresponding author: \href{mailto:benedp@usi.ch}{benedp@usi.ch}}}
\author[1,2]{Simone Riva}
\author[1,2]{Luca Belluzzi}
\author[3,4]{Stefano Serra-Capizzano}

\affil[1]{Istituto ricerche solari Aldo e Cele Dacc\`o (IRSOL), Universit\`a della Svizzera italiana, Locarno, Switzerland}
\affil[2]{Euler Institute, Universit\`a della Svizzera italiana, Lugano, Switzerland}
\affil[3]{Insubria University, Como, Italy}
\affil[4]{Uppsala University, Uppsala, Sweden}
\begin{document}
\maketitle

\begin{abstract}
We consider a multidimensional polychromatic radiative transfer (RT) problem, accounting for scattering processes in a general form, i.e. anisotropic (dipole) scattering with partial frequency redistribution.
Given a discrete ordinates ($S_N$) discretization, we report the corresponding matrix structures, depending on model and discretization parameters.
Despite the possibly dense nature of these matrices, the use of Krylov methods is effective (especially in the matrix-free context) and robust. We propose a theoretical analysis, using the spectral tools of the \textit{symbol} theory, explaining why Krylov convergence is robust w.r.t. all the discretization parameters, even in the unpreconditioned case.
In fact, the compactness of the continuous operators used in the modeling leads to zero-clustered dense matrix sequences plus identity, so that the clustering at the unity of the spectra is deduced. Numerical experiments confirm the theoretical results, which have a direct application, for example, in the simulation of radiative transfer in stellar atmospheres, a key problem in astrophysical research. In general, we demonstrate that optimal scaling  with respect to RT discretization parameters is expected for Krylov solution strategies.
\end{abstract}

\section{Introduction}
Radiative transfer (RT) describes how electromagnetic radiation is altered, through absorption, emission, and scattering, as it propagates through a participating medium.
It plays a fundamental role in numerous scientific and engineering fields, including astrophysics, climate science, heat transfer, nuclear engineering, biomedical imaging, and materials science.
More specifically, RT is central to remote sensing applications, where it enables the interpretation of information encoded in electromagnetic radiation collected by observational instruments.
When considering systems out of local thermodynamic equilibrium, the RT problem consists in finding a self-consistent solution of two sets of equations describing the interaction between matter and radiation: the RT equation for the radiation field and the statistical equilibrium (SE) equations for the medium \cite{landi2016}.
The resulting RT problem is generally nonlocal and nonlinear, is integro-differential by nature, and can be highly coupled and high-dimensional (up to seven dimensions in the most general time-dependent case). 
As a result, its numerical solution in realistic scenarios is often challenging, requiring the use of efficient numerical methods.

Numerical algorithms for radiative transfer can be broadly classified as stochastic (Monte Carlo, cf. \cite{noebauer2019monte}) or deterministic. Among deterministic approaches, the discrete ordinates method (DOM), also known as the $S_N$ method, is arguably the most extensively studied and widely used technique for angular discretization when both accuracy and computational efficiency are required \cite{balsara2001fast}. In the DOM framework, the angular domain is discretized into a finite set of directions, or \textit{ordinates}, along which the RT equation is solved using an appropriate spatial discretization.
In this work, we do not consider other angular discretizations, such as the spherical harmonics method, which has different mathematical properties and numerical challenges.
%In particular, the complexity of the $P_N$ method increases rapidly with dimensionality and expansion order, requiring a careful selection for optimal performance and accuracy.
Concerning the spatial discretization, the RT equation can be discretized using the finite difference method or using exponential integrators, e.g. in the context of short/long-characteristics methods, or using the finite volume/element methods \cite{bardin2025accelerated,olivier2025consistent}.
Characteristic methods are typically employed on regular, structured meshes and for relatively simple geometries, whereas finite volume (FVM) and finite element (FEM) methods are better suited for complex geometries. The main drawback of the latter methods lies in the cost of assembling discrete systems, particularly in terms of memory usage. This issue becomes especially severe in three-dimensional polychromatic problems, which involve six dimensions, and matrix assembly becomes often prohibitive. In such cases, matrix-free or semi–matrix-free \cite{jolivet2021deterministic} approaches must be adopted.

This work focuses on linear radiative transfer problems. Once discretized, such problems can be efficiently solved using Krylov subspace methods such as preconditioned GMRES or BiCGStab \cite{anusha2009preconditioned,badri2019preconditioned,benedusi2022numerical}. The convergence analysis of these iterative methods is closely related to the spectral clustering of the resulting matrix sequences, which can be formally investigated using tools from the \textit{symbol theory}. In this framework, the ``symbol'' can be interpreted as a function describing the asymptotic eigenvalue or singular value distribution of increasingly large discrete operators obtained via grid refinement. Such an analysis provides precise insight into the spectral properties of the operator sequences and, importantly, is often predictive already at moderate problem sizes, not only in the asymptotic limit.

In particular, for a general linear RT problem, discretized with DOM and a long-characteristics method, we show that Krylov solvers are expected to exhibit robustness with respect to all the discretization parameters in 6D, even in the absence of preconditioning. In this sense, we can refer to a “free lunch”: robustness, and hence optimal scaling with respect to problem size, is essentially obtained for free. Naturally, additional preconditioning can still be employed to further reduce the time to solution \cite{dolz2022robustly}. These findings complement classical results on the spectrum of continuous integral operators, whose compactness implies that zero is the only possible accumulation point of the spectrum.

This work is organized as follows: in Section~\ref{sec:model}, we present a general linear RT model problem. In Section~\ref{sec:disc}, after some simplifying assumptions, we describe its discretization and the corresponding matrix sequences. The extension (back to a more general case) is reported at the end of the section. In  Section~\ref{sec:symbol}, we describe the theoretical tools and present the main results, which are validated in Section~\ref{sec:exp}, via multiple numerical experiments. Finally, in Section~\ref{sec:concl}, we report the concluding remarks.

\section{Model}\label{sec:model}
Given a spatial domain $D\subset \mathbb{R}^d$, a frequency domain $F\subset \mathbb{R}_+$, and $\mathbb{S}^{q}$ the $q$-dimensional unit sphere (typically $q\in\{1,2\}$), we define $\mathcal{D}= D\times \mathbb{S}^{q} \times F$ and consider the polychromatic stationary radiative transfer equation
\begin{equation}\label{eq:RT}
\nabla_{\bm{\Omega}} I(\bm{x},\bm{\Omega},\nu)=-\chi(\bm{x},\bm{\Omega},\nu)I(\bm{x},\bm{\Omega},\nu) + \varepsilon(\bm{x},\bm{\Omega},\nu), \quad \mathrm{for} \quad (\bm{x},\bm{\Omega},\nu)\in \mathcal{D},
\end{equation}
where $\nabla_{\bm{\Omega}}$ is the directional derivative along $\bm{\Omega}$, the unknown $I:\mathcal{D} \to \mathbb{R}_+$ is the radiation intensity, 
%and $\varepsilon:\mathcal{D} \to \mathbb{R}$ is the emissivity field, which reads
%\cluca{$\chi:\mathcal{D} \to \mathbb{R}_+$ is the extinction coefficient (or opacity), and $\varepsilon:\mathcal{D} \to \mathbb{R}_+$ is the emission coefficient.}
while $\chi:\mathcal{D} \to \mathbb{R}_+$ and $\varepsilon:\mathcal{D} \to \mathbb{R}_+$ are the medium's extinction coefficient (or opacity) and emission coefficient, respectively.
%-----------
The RT equation~\eqref{eq:RT} is a linear differential equation, to be closed by suitable boundary conditions, e.g.
\begin{equation}\label{eq:bc}
    I(\bm{x},\bm{\Omega},\nu)=I_{\mathrm{in}}(\bm{x},\nu) \quad \mathrm{for} \quad \bm{x}\in\partial D,\bm{\Omega}\in S_{\Gamma},\nu\in F,
\end{equation}
where $S_{\Gamma}\subset S$ is a set of directions (usually incoming directions in $\partial D$, i.e. satisfying $\bm{\Omega}\cdot\bm{n}<0$, given the outgoing normal $\bm{n}$).
Given a direction $\bm{\Omega}$,  Equation~\eqref{eq:RT} is intrinsically one-dimensional. Introducing the coordinate $s\in[0,L]$ along $\bm{\Omega},$ and fixing $\nu$, we can write \eqref{eq:RT} as
\begin{equation}\label{eq:RTs}
    \frac{\mathrm{d}}{\mathrm{d}s}\,I(s) =-\chi(s) I(s) +\varepsilon(s) .
\end{equation}
% \cluca{Defining the optical depth scale, for the considered direction and frequency, ${\rm d}\tau = -\chi(s) {\rm d}s$ and the source function $S = \varepsilon/\chi$, the RT equation takes the form}
% $$\frac{\mathrm{d}}{\mathrm{d} \tau} \, I(\tau) = I(\tau) - S(\tau).$$ 
In the context of radiative transfer, the procedure to solve Equation~\eqref{eq:RT}, i.e. solving~\eqref{eq:RTs} in each direction, given $\chi$ and $\varepsilon$, is known as {\it formal solution}.
%-------------

The extinction and emission coefficients $\chi$ and $\varepsilon$ depend on the state of the medium. Out of local thermodynamic equilibrium, this state has to be evaluated by solving the SE equations, which account for the interaction of the medium with the radiation field.
%which in turn depends on the radiation field via the statistical equilibrium (SE) equations.
%The extinction coefficient can be written as $\chi(\bm{x},\bm{\Omega},\nu)=[\kappa(\bm{x})+\sigma(\bm{x})]\phi(\bm{x},\bm{\Omega},\nu)$, where $\phi$ is the absorption profile (normalized to unity with respect to frequency, for any $\bm{x}$ and $\bm{\Omega}$), while $\kappa$ and $\sigma$ are the frequency-integrated extinction coefficients for (true) absorption and scattering, respectively.
Hereafter, we will focus on media for which the SE equations have an analytic solution. 
In this case, the emission coefficient %can be explicitly\cluca{[!]} related to the radiation field and 
can be expressed as:
\begin{equation}\label{eq:emissivity}
\varepsilon(\bm{x},\bm{\Omega},\nu)=\frac{\sigma(\bm{x})}{s_d}\oint_{\mathbb{S}^{q}}\int_F\Phi(\bm{x},\bm{\Omega},\bm{\Omega}',\nu,\nu')I(\bm{x},\bm{\Omega}',\nu')\,\mathrm{d}\nu'\mathrm{d}\bm{\Omega}'+\varepsilon_t(\bm{x},\bm{\Omega},\nu),
\end{equation}
where the integral term describes scattering contributions, while $\varepsilon_t:\mathcal{D} \to \mathbb{R}_+$ is a thermal emission source term.
The quantity $\sigma$ is the frequency-integrated extinction coefficient for scattering processes and $s_d\in\mathbb{R}$ is the surface of $\mathbb{S}^{q}$. The scattering kernel $\Phi:D\times \mathbb{S}^{q} \times \mathbb{S}^{q} \times F \times  F \to \mathbb{R}$, also known as redistribution function, encodes the SE equations, and represents the joint probability of a photon $(\bm{\Omega}',\nu')$ to be absorbed and a photon $(\bm{\Omega},\nu)$ to be emitted in a scattering process.
%\cluca{This formalism is very general and allows} considering anisotropic, angle-dependent scattering.
%where $s_d\in\mathbb{R}$ is the surface of $\mathbb{S}^{q}$.
%The physical properties of the domain are encoded in the extinction (or opacity) coefficient $\chi:\mathcal{D} \to \mathbb{R}_+$, also expressed as $\chi(\bm{x},\bm{\Omega},\nu)=[\kappa(\bm{x},\Omega)+\sigma(\bm{x})]\phi(\nu)$, i.e. as the sum of an absorption coefficient ($\kappa$) and a dimensional scattering coefficient $\sigma:D \to \mathbb{R}$, scaled by an absorption profile $\phi:F \to \mathbb{R}_+$.
%
% \begin{equation}
% \frac{1}{s_d\cdot s_d}\oint_{\mathbb{S}^{q}}\int_F\oint_{\mathbb{S}^{q}}\int_F\Phi(\bm{x},\bm{\Omega},\bm{\Omega}',\nu,\nu')\,\mathrm{d}\nu'\mathrm{d}\bm{\Omega}'\mathrm{d}\nu\mathrm{d}\bm{\Omega}=1
% \end{equation}
%
%The process of scattering is described via the scattering kernel $\Phi:D\times \mathbb{S}^{q} \times \mathbb{S}^{q} \times F \times  F \to \mathbb{R}$. Finally, $\varepsilon_t:\mathcal{D} \to \mathbb{R}$ is a source term, for example given by thermal emissivity. We remark that the scattering kernel, also known as redistribution function, represents the probability of a photon $(\bm{\Omega}',\nu')$ to be absorbed and the photon $(\bm{\Omega},\nu)$ to be emitted. Thus, 
Since photons are conserved during scattering, the following normalization must hold
\begin{equation*}
    \frac{1}{(s_d)^2} \oint_{\mathbb{S}^{q}} \mathrm{d}\bm{\Omega} \oint_{\mathbb{S}^{q}} \mathrm{d}\bm{\Omega}' \int_F \mathrm{d}\nu \int_F  \mathrm{d}\nu' \, \Phi(\bm{x},\bm{\Omega},\bm{\Omega}',\nu,\nu') = 1.
    %\quad \text{for all} \quad (\bm{x},\bm{\Omega},\nu)\in \mathcal{D}.
\end{equation*}
This formalism allows considering anisotropic (dipole) scattering, as well as partial frequency redistribution (PRD) effects. 

The extinction coefficient in \eqref{eq:RT} can be written as $\chi(\bm{x},\bm{\Omega},\nu)=[\kappa(\bm{x})+\sigma(\bm{x})]\phi(\bm{x},\bm{\Omega},\nu)$, where $\phi$ is a normalized absorption profile\footnote{We note that the dependence of the absorption profile on the direction $\bm{\Omega}$ is a consequence of bulk motions in the medium.}, while $\kappa$ is the frequency-integrated extinction coefficient for (true) absorption.
%Similarly to the emission coefficient, also the extinction coefficient $\chi$ depends on the radiation field $I$, via the SE equations, and that this dependence is, in general, nonlinear.
The dependence of the extinction coefficient on the radiation field, via the SE equations is, in general, nonlinear.
We assume that $\chi$ is not affected by the radiation field, and it is a known input quantity of the problem.
Under this assumption, the RT problem is fully described by the linear equations \eqref{eq:RT} and \eqref{eq:emissivity}, for which a consistent solution has to be found. We note that the more general nonlinear problem is commonly treated via linearization; therefore, the results presented in this work remain relevant. A similar consideration holds for the time-dependent problem, which is frequently recast as a sequence of stationary problems.

Equations~\eqref{eq:RT}-\eqref{eq:emissivity} can be expressed from an operator prospective. Given some suitable function spaces $X_I$ and $X_\varepsilon$, so that $I\in X_I$ and $\varepsilon\in X_\varepsilon$, we define a transfer operator $\mathcal{T}$ and a scattering operator $\mathcal{S}$, respectively
\begin{equation}\label{eq:LS}
    \mathcal{T}:X_\varepsilon\to X_I, \qquad  \mathcal{S}:X_I \to X_\varepsilon.
\end{equation}
The operator $\mathcal{T}$ maps an emissivity field $\varepsilon$ to a radiation intensity $I$, via the solution of \eqref{eq:RT}, while $\mathcal{S}$ maps a radiation field $I$ to an emissivity field, evaluating the scattering integral in \eqref{eq:emissivity}. 
%In the context of radiative transfer, $\mathcal{T}$ is also known as \textit{formal solver}, and its application to an emissivity field \textit{formal solution}, often denoted with the greek letter $\Lambda$.
%Introducing the one-dimensional coordinate $s\in[0,S]$, along the direction $\bm{\Omega}$ and fixing $\nu$, we can write \eqref{eq:RT} as $$\frac{\mathrm{d}}{\mathrm{d}s}\,I(s) =-\chi(s) I(s) +\varepsilon(s),$$
Observing that the RT equation in the form \eqref{eq:RTs} has solution
%with solution, defining the optical depth, $\tau(s_0,s)=-\int_{s_0}^s\chi(s')\,\mathrm{d}s'$,
\begin{equation}\label{eq:rt_integral}
    I(s) =  \int_0^se^{\tau(s',s)}\varepsilon(s')\,\mathrm{d}s' + I_{\mathrm{in}}(0)\cdot e^{\tau(0,s)},
\end{equation}
with $\tau(s_0,s)=-\int_{s_0}^s\chi(s')\,\mathrm{d}s'$, we see that also the transfer operator is
% \begin{equation}\label{eq:rt_integral}
%     I(s) =  \int_0^se^{-\int_{s'}^s\chi(t)\,\mathrm{d}t}\varepsilon(s')\,\mathrm{d}s' + I_{\mathrm{in}}(0)\cdot e^{-\int_{s_0}^s\chi(t)\,\mathrm{d}t},
% \end{equation}
%which also is 
an integral mapping between emissivity and intensity, given an exponential kernel plus a boundary term.
The self-consistent  solution of \eqref{eq:RT} and \eqref{eq:emissivity} can be expressed as the fixed point of the composed operator $\mathcal{T}\mathcal{S}$ (which encodes a triple integral)
\begin{equation*}
    \mathcal{T}(\mathcal{S}(I))=I.
\end{equation*}
It is then natural to look for the solution of \eqref{eq:RT} and \eqref{eq:emissivity} using a fixed point iteration (also known as lambda or source iteration in the RT context). We refer to \cite{egger2014lp} for a general result on the spectral radius of the linear operator $\mathcal{T}\mathcal{S}$, showing that the corresponding fixed point iteration converges in $L_p$ (for any $1\leq p\leq \infty$) if $\chi/\sigma <1$ via a standard contraction argument.
Since the fixed point iteration is often too slow for practical use, more efficient (preconditioned) Krylov methods can be adopted, which can exploit favorable spectral distributions \cite{beckermann2001superlinear,kuijlaars2006convergence}. In fact, the convergence of such methods is closely related to the clustering of the spectrum of the discretized $\mathcal{T}\mathcal{S}$ operator, which will be the object of this study, generalizing previous results related to its spectral radius.

We note that for any continuous kernel defined on a bounded subset of ${\mathbb R}^{d}$, the corresponding integral operator is compact and its spectrum forms a bounded sequence having zero as unique accumulation point. The approximation of an integral operator by a quadrature rule, e.g. $\mathcal{S}$ according to \eqref{eq:emissivity}, results in a linear operator represented by a matrix $\Sigma_N$ of size $N$. In accordance with the compact character of the continuous operator $\mathcal{S}$, as discussed in \cite{al2014singular}, the singular values and eigenvalues of $\{\Sigma_N\}$ are clustered at zero as $N\to\infty$. Moreover, the cluster is strong, in perfect agreement with the compactness of the continuous operator. We will present the discretization of $\mathcal{T}$ and $\mathcal{S}$ in the following section, while the notion of eigenvalues distribution of of matrix sequences and their strong clustering will be discussed in more detail in Section~\ref{sec:preliminaries}.

\section{Discretization and operator assembly}\label{sec:disc}
The discretization will initially be presented in the case of a one-dimensional (1D) plane-parallel spatial domain with cylindrical symmetry. 
%This simplification is also informative since Equation~\eqref{eq:RT}, once $\bm{\Omega}$ is fixed, is essentially a 1D equation, and the multi-dimensional case can be reduced to a set of independent 1D problems. 
In a 1D plane-parallel setting, all physical quantities only vary along a given direction (hereafter with coordinate $z$) and are constant on the planes perpendicular to it. 
Thanks to the assumption of cylindrical symmetry, the direction $\bm{\Omega}$ is fully described by the inclination $\theta\in[0,\pi]\setminus \{\pi/2\}$ with respect to the vertical, or equivalently by $\mu=\cos{\theta}$, with $\mu>0$  corresponding to the emerging (i.e. observable) radiation.
%In 1D, the spatial dependence is encoded by the optical depth $\tau\in [\tau_{\mathrm{surf}},\tau_{\mathrm{deep}}]\subset\mathbb{R}$; given the optical depth conversion\footnote{We consider $\mathrm{d}\tau(\mu,\nu)=-\chi(s,\mu,\nu)\mathrm{d}s=-\phi(\nu)\chi(s,\mu)\mathrm{d}s$.}, the transfer equation~\eqref{eq:RT} can then be rewritten as
% We consider a right-handed Cartesian reference system with the $z$-axis directed along the vertical.
Introducing the optical depth scale $\mathrm{d}\tau(\mu,\nu) = -\chi(z, \mu, \nu)\mathrm{d}z/\mu$, the RT equation \eqref{eq:RT} can be written as
\begin{align}
	\frac{\mathrm{d}}{\mathrm{d}\tau}
	I(\tau,\mu,\nu) & =I(\tau,\mu,\nu) - S(\tau,\mu,\nu), \label{RT_tau} \quad \mathrm{for}\quad \tau\in[\tau_\mathrm{surf},\tau_{\mathrm{deep}}],\,\mu\in[0,1]/\{0\},\nu\in F,\\
    I(\tau_{\mathrm{deep}},\mu,\nu) & =I_\mathrm{in}(\tau_{\mathrm{deep}},\nu) \quad \text{for} \quad \mu>0, \nonumber\\
    I(\tau_{\mathrm{surf}},\mu,\nu) & =I_\mathrm{in}(\tau_{\mathrm{surf}},\nu) \quad \text{for} \quad \mu<0, \nonumber
\end{align}
where $S=\varepsilon/\chi$ is the so-called \textit{source function}.
% %
% \begin{align}
% 	\frac{\mathrm{d}}{\mathrm{d}\tau
% 	I(\tau,\mu,\nu) & = I(\tau,\mu,\nu) - S(\tau,\mu,\nu), \label{RT_tau} \\
%     I(\tau_{\mathrm{deep}},\mu,\nu) & =I_\mathrm{in}(\tau_{\mathrm{deep}},\nu) \quad \text{for} \quad \mu>0, \nonumber\\
%     I(\tau_{\mathrm{surf}},\mu,\nu) & =I_\mathrm{in}(\tau_{\mathrm{surf}},\nu) \quad \text{for} \quad \mu<0, \nonumber
% \end{align} 
% %
Despite this simplifying assumption, the core integro-differential structure of the radiative transfer problem, i.e. the the interplay between $\mathcal{T}$ and $\mathcal{S}$ operators, is maintained. We will cover the 3D generalization of Equation~\eqref{RT_tau} in Section~\ref{sec:3D} using the long-characteristics method, simply considering multiple independent ODEs, covering the 3D space.

\subsection{Plane-parallel discretization}
We consider a discrete ordinate method (DOM) to obtain a discrete version of equations \eqref{eq:RT} and \eqref{eq:emissivity}.  We refer to \cite{richling2001radiative,meinkohn2002radiative,castro2015spatial,badri2018high} for analogous derivations using finite elements. We consider a discrete grid in space with $N_s$ nodes \begin{equation*}
     \tau_{\mathrm{surf}} = \tau_1 < \tau_2 < \ldots < \tau_{N_s-1} < \tau_{N_s} = \tau_{\mathrm{deep}},
\end{equation*}
where $\tau_i=\tau_i(\mu,\nu)$, and a discrete grid with $N_\Omega$ (even) directions
\begin{equation*}
     -1 \leq \mu_1 < \mu_2 < \ldots < 0 < \ldots < \mu_{N_\Omega -1} < \mu_{N_\Omega} \leq 1.
\end{equation*}
Similarly, the frequency domain $F$ is discretized with $N_\nu$ distinct frequencies $\{\nu_i\}_{i=1}^{N_\nu}\subset F$. Given $N=N_sN_\Omega N_\nu$ the total number of degrees of freedom, we consider the collocation vector $\mathbf{X} \in \mathbb{R}^N$ collecting
all the discrete values of the continuous field $X(\tau,\mu,\nu)$ in lexicographic ordering , namely,
\begin{align*}
 \mathbf{X} = & [
X(\tau_1,\mu_1,\nu_1),
X(\tau_1,\mu_1,\nu_2),\ldots,
X(\tau_1,\mu_1,\nu_{N_\nu}),
X(\tau_1,\mu_2,\nu_1),
X(\tau_1,\mu_2,\nu_2),\ldots,\\
& X(\tau_1,\mu_2,\nu_{N_\nu}),\ldots,
X(\tau_1,\mu_{N_\Omega},\nu_{N_\nu}),
X(\tau_2,\mu_1,\nu_1),\ldots,
X(\tau_{N_s},\mu_{N_\Omega},\nu_{N_\nu})]^T.
 \end{align*}
For notational simplicity we also define the \textit{ray} $\bm{r}=(\mu,\nu)\in[-1,1]\times F$, with $N_r=N_\Omega N_\nu$ the total number of rays, and the ordering
$$\bm{r}_k=(\mu_{\left \lceil{k/N_\nu}\right \rceil},\nu_{k\bmod N_\nu}) \quad \text{for} \quad k=1,\ldots,N_r,$$
so that we can write $\mathbf{X}$ more compactly:
\begin{align*}
\mathbf{X} = & [
X(\tau_1,\bm{r}_1),
X(\tau_1,\bm{r}_2),\ldots,
X(\tau_1,\bm{r}_{N_r}),
X(\tau_2,\bm{r}_{1}),
\ldots,
X(\tau_2,\bm{r}_{N_r}),\ldots,
X(\tau_{N_s},\bm{r}_{N_r})]^T.
\end{align*}
We also define the vector
$\mathbf{X}_i=\left[X(\tau_i,\bm{r}_1),\ldots,X(\tau_i,\bm{r}_{N_r})\right]^T \in\mathbb R^{N_r}
$
which collects the entries of $\mathbf{X}$ corresponding to the spatial point $\tau_i$, such that
$\mathbf{X}=\left[\mathbf{X}_1^T,\ldots,\mathbf{X}_{N_s}^T\right]^T$.
We also consider a \textit{ray-based} permutation of $\mathbf{X}$, namely
\begin{align*}
\widetilde{\mathbf{X}} = & [
X(\tau_1,\mu_1,\nu_1),
X(\tau_2,\mu_1,\nu_1),\ldots,
X(\tau_{N_s},\mu_1,\nu_1),
X(\tau_1,\mu_1,\nu_2), \\
& X(\tau_2,\mu_1,\nu_2),\ldots,\ldots,X(\tau_{N_s},\mu_{N_\Omega},\nu_{N_\nu})]^T,
\end{align*}
so that $\widetilde{\mathbf{X}}=P\mathbf{X}$, given a permutation matrix $P\in\mathbb{R}^{N\times N}$. We define the vector corresponding to a single ray $\bm{r_k}$:
$$\widetilde{\mathbf{X}}_k=[X(\tau_1,\bm{r}_k),\ldots,X(\tau_{N_s},\bm{r}_k)]^T\in\mathbb{R}^{N_s},$$
so that the full vector can be written as
$$ \widetilde{\mathbf{X}} = [\widetilde{\mathbf{X}}_1^T,\widetilde{\mathbf{X}}_2^T,\ldots,\widetilde{\mathbf{X}}_{N_r}^T]^T. $$
We then consider the discrete counterparts of the radiation intensity and the emissivity field (or source function), $(I,S)$ in Equation~\eqref{RT_tau}, and introduce the corresponding collocation vectors with space-based and ray-based orderings, i.e. $\mathbf{I},\mathbf{S},\widetilde{\mathbf{I}},\widetilde{\mathbf{S}}\in\mathbb{R}^N$.
Since both equations \eqref{eq:RT} and \eqref{eq:emissivity} are linear, we can write them in a compact algebraic form:
\begin{align}
    & \mathbf{I}=\Lambda_N\mathbf{S}+\mathbf{I}_\mathrm{in}, \label{eq:RT_algebraic_form1} \\
    & \mathbf{S}=\Sigma_N\mathbf{I}+\mathbf{t}, \label{eq:RT_algebraic_form2}
\end{align}
with $\Lambda_N,\Sigma_N\in\mathbb{R}^{N\times N}$. Equation \eqref{eq:RT_algebraic_form1} encodes the numerical solution of \eqref{eq:RT} (cf. Equation~\eqref{eq:rt_integral}), given an input emissivity vector $\bm{\varepsilon}$ and the boundary data %
$$ \mathbf{I}_\mathrm{in} = [\mathbf{I}_{\mathrm{in},1}^T,\bm{0},\mathbf{I}_{\mathrm{in},N_s}^T]^T\in\mathbb{R}^N,$$
where only the first and last $N_\nu N_\Omega/2$ elements are possibly nonzero. The structure of the operator $\Lambda_N$, which depends on the extinction coefficient $\chi$ and on the numerical method used to solve \eqref{RT_tau}, will be discussed in Section~\ref{sec:L}. Similarly,  equation \eqref{eq:RT_algebraic_form2} encodes the computation of the scattering integral in \eqref{eq:emissivity}, given an input emissivity $\mathbf{I}$ and the collocation vector $\mathbf{t}\in\mathbb{R}^N$ discretizing $\varepsilon_t$.
Intuitively, the operator $\Lambda_N$ couples the radiation at different spatial points, while $\Sigma_N$ couples different directions and frequencies (i.e. rays), as encoded in the scattering kernel $\Phi$ in \eqref{eq:emissivity}.
We notice that the transformations defined in equations \eqref{eq:RT_algebraic_form1} and \eqref{eq:RT_algebraic_form2} are the discrete counterparts of the continuous maps defined in \eqref{eq:LS}.

Solving the considered discrete radiative transfer problem corresponds to finding the vector $\mathbf{I}$ so that \eqref{eq:RT_algebraic_form1} and \eqref{eq:RT_algebraic_form2} are both satisfied. By substitution, we obtain the following linear system to be solved
\begin{equation}\label{eq:RT_sys_algebraic}
    (I\!d_N-\Lambda_N\Sigma_N)\mathbf{I}=\Lambda_N\mathbf{t}+\mathbf{I}_\mathrm{in},
\end{equation}
or, more compactly,
\begin{equation*}
    A_N\mathbf{I}=\mathbf{b},
\end{equation*}
with $\mathbf{b}=\Lambda_N\mathbf{t}+\mathbf{I}_\mathrm{in}$ and $A=I\!d_N-\Lambda_N\Sigma_N$.
We note that \eqref{eq:RT_sys_algebraic} can be seen as a discretized Fredholm integral equation of the second kind.

%Transport operator
\subsection{Transfer operator $\Lambda_N$}\label{sec:L}
First, we explicitly derive the entries of the  matrix $\Lambda_N$ using the implicit Euler method to integrate the transfer equation~\eqref{RT_tau} (i.e. a right-hand rectangular quadrature rule for \eqref{eq:rt_integral}), as a convenient example.
A similar derivation could be performed to any other formal solver. Applying the implicit Euler stencil to \eqref{RT_tau}, we obtain, for incoming and outgoing directions $\bm{r}=(\mu,\nu)$
\begin{align}
    I(\tau_{i+1},\bm{r}) & = \frac{I(\tau_{i},\bm{r}) +\Delta \tau_i(\mu,\nu)S(\tau_{i+1},\bm{r})}{1+\Delta \tau_i(\bm{r})},\qquad
    \text{ for } \quad \mu<0 \quad \text{ and } \quad  i = 1,\ldots,N_{s}-1,\label{ie_1}\\
    I(\tau_{i-1},\bm{r}) & = \frac{I(\tau_i,\bm{r}) +\Delta \tau_{i-1}(\bm{r})S(\tau_{i-1},\bm{r})}{1+\Delta \tau_{i-1}(\bm{r})},  \qquad
    \text{ for } \quad \mu>0 \quad \text{ and } \quad  i = 2,\ldots,N_s, \label{ie_2}
\end{align}
where
\begin{equation*}
\Delta \tau_{i}(\bm{r})=\Delta \tau_{i}(\mu,\nu)=\tau_{i+1}(\mu,\nu)-\tau_i(\mu,\nu).
\end{equation*}
We notice that, since $\Delta\tau_i$ is positive for all $i$, if $S=0$, i.e. only absorption is present, the intensity is exponentially decaying while the radiation is traversing the spatial domain.
Moreover, different directions and frequencies
are decoupled
in~\eqref{ie_1} and~\eqref{ie_2}, as in the continuous counterpart~\eqref{eq:RT}. Providing the initial conditions $I(\tau_1,\mu,\nu)=I_{\text{in}}(\tau_1,\nu)$ for all $\mu<0$ and $I(\tau_{N_s},\mu,\nu)=I_\mathrm{in}(\tau_{N_s},\nu)$ for $\mu>0$,
we recursively apply~\eqref{ie_1} and~\eqref{ie_2} and obtain
\begin{align}
    I(\tau_i,\bm{r})  & = f_i(\bm{r}) I_{\text{in}}(\tau_1,\nu) + \sum_{j=2}^i f_{i,j}(\bm{r})S(\tau_j,\bm{r}) \qquad \text{ for } \quad \mu<0 \quad \text{ and } \quad  i = 2,\ldots,N_s,\label{ie_mono1}\\
    I(\tau_i,\bm{r})  & = h_i(\bm{r}) I_{\text{in}}(\tau_{N_s},\nu) + \sum_{j=i}^{N_s-1} h_{i,j}(\bm{r})S(\tau_j,\bm{r}) \qquad  \text{ for } \quad \mu>0 \quad \text{ and } \quad  i = 1,\ldots,N_s-1,\label{ie_mono2}
\end{align}
with
\begin{align}
    f_{i,j}(\bm{r}) & = \frac{\Delta \tau_{j-1}(\bm{r})}{\prod_{l=j}^{i}\left[  1 + \Delta \tau_{l-1}(\bm{r})\right]} \quad \text{and} \quad f_i(\bm{r}) = \frac{1}{\prod_{l=1}^{i-1}  \left[1 + \Delta \tau_{l}(\bm{r})\right]},\label{ie_mono3}\\
    h_{i,j}(\bm{r}) & = \frac{\Delta \tau_{j}(\bm{r})}{\prod_{l=i}^{j}  \left[1 + \Delta \tau_{l}(\bm{r})\right]} \quad \text{and} \quad h_i(\bm{r}) = \frac{1}{\prod_{l=i}^{N_s-1}  \left[1 + \Delta \tau_{l}(\bm{r})\right]}.\label{ie_mono4}
\end{align}
For any $\bm{r}=(\mu,\nu)$ with positive $\mu$, equation~\eqref{ie_mono1} can be arranged in the following matrix form
% $F\in\mathbb R^{N_s\times N_s}$, namely,
%
$$ \begin{bmatrix}
I(\tau_1,\bm{r}) \\
I(\tau_2,\bm{r})\\
%I(\nu,\mu,\tau_3) \\
\vdots \\
I(\tau_{N_s},\bm{r}) \\
\end{bmatrix} =
\begin{bmatrix}
0 & & &  \\\
& f_{2,2}(\bm{r}) & &  \\
& \vdots & \ddots &  \\
 & f_{N_s,2}(\bm{r})   & \cdots   & f_{N_s,N_s}(\bm{r})\\
\end{bmatrix}
\begin{bmatrix}
S(\tau_1,\bm{r}) \\
S(\tau_2,\bm{r}) \\
%\varepsilon(\mu,\tau_3) \\
\vdots \\
S(\tau_{N_s},\bm{r}) \\
\end{bmatrix} +
I_{\text{in}}(\tau_1,\nu)
\begin{bmatrix}
1 \\
f_2(\bm{r}) \\
%\varepsilon(\mu,\tau_3) \\
\vdots \\
f_{N_s}(\bm{r}) \\
\end{bmatrix},$$
or, more compactly
\begin{equation}\label{eq:lamda_down}
  \bm{I}(\bm{r}) =  \Lambda^{\downarrow}(\bm{r})\bm{S}(\bm{r}) + \bm{I}_{\text{in}}^{\downarrow}(\bm{r}),\quad \mathrm{for}\quad\mu<0,
\end{equation}
where $\Lambda^{\downarrow}\in\mathbb{R}^{N_s\times N_s}$ and $\bm{I},\bm{S},\bm{I}_{\text{in}}^{\downarrow}\in\mathbb{R}^{N_s}$ follow directly from the above matrix equation.
Analogously, equation~\eqref{ie_mono2} leads to
$$ \begin{bmatrix}
I(\tau_1,\bm{r}) \\
I(\tau_2,\bm{r})\\
\vdots \\
I(\tau_{N_s},\bm{r}) \\
\end{bmatrix} =
\begin{bmatrix}
h_{1,1}(\bm{r}) & \cdots & \cdots h_{1,N_s- 1}(\bm{r}) & \\
& \ddots &  \vdots &\\
 & & h_{N_s-1,N_s-1}(\bm{r}) &\\
 & & & 0
\end{bmatrix}
\begin{bmatrix}
S(\tau_1,\bm{r}) \\
S(\tau_2,\bm{r}) \\
\vdots \\
S(\tau_{N_s},\bm{r}) \\
\end{bmatrix} +
I_{\text{in}}(\tau_{N_s},\nu)
\begin{bmatrix}
h_1(\bm{r}) \\
\vdots \\
h_{N_s-1}(\bm{r}) \\
1 \\
\end{bmatrix},$$
or, more compactly
\begin{equation}\label{eq:lamda_up}
    \bm{I}(\bm{r}) =  \Lambda^{\uparrow}(\bm{r})\bm{S}(\bm{r}) + \bm{I}_{\text{in}}^{\uparrow}(\bm{r}),\quad \mathrm{for}\quad\mu>0.
\end{equation}
Combining \eqref{eq:lamda_down} and \eqref{eq:lamda_up}, including the discretizations over directions and frequencies, conveniently using the ray-based ordering, with
\[
\Lambda_{[k]} =
     \begin{cases}
       \Lambda^{\downarrow}(\bm{r}_k) &\quad \text{for} \quad k\leq N_r/2 \quad (\text{i.e.}\,\mu<0),\\
       \Lambda^{\uparrow}(\bm{r}_k) &\quad \text{for} \quad k> N_r/2 \quad (\text{i.e.}\,\mu>0),
     \end{cases}
\]
for $k=1,\ldots,N_r$, and similarly for $(\widetilde{\mathbf{I}}_\text{in})_k$, we can write\footnote{To avoid confusion, we use the following notation along this work: we write $A_N$ for a square matrix of size $d_N$ (with $d_N=N$ for RT operators) and $A_{[k]}$ to denote the $k$th block index for block-diagonal matrices.}
\begin{equation}\label{eq:lambda_block}
\begin{bmatrix}
\widetilde{\mathbf{I}}_1 \\
\widetilde{\mathbf{I}}_2 \\
\vdots \\
\mathbf{I}_{N_r}
\end{bmatrix} =
\begin{bmatrix}
\Lambda_{[1]} & &  \\
& \Lambda_{[2]} &  \\
& & \ddots &  \\
& & &\Lambda_{[N_r]}
 \end{bmatrix}
 \begin{bmatrix}
\widetilde{\mathbf{S}}_1 \\
\widetilde{\mathbf{S}}_2 \\
\vdots \\
\mathbf{S}_{N_r}
\end{bmatrix}
+
\begin{bmatrix}
(\widetilde{\mathbf{I}}_\text{in})_1 \\
(\widetilde{\mathbf{I}}_\text{in})_2 \\
\vdots \\
(\widetilde{\mathbf{I}}_\text{in})_{N_r}
\end{bmatrix},
\end{equation}
or, more compactly,
\begin{equation}
    \widetilde{\mathbf{I}} =  \widetilde{\Lambda}_N\widetilde{\mathbf{S}} + \widetilde{\mathbf{I}}_{\text{in}},
\end{equation}
which is a permuted version of \eqref{eq:RT_algebraic_form1},
with $\widetilde{\Lambda}=P\Lambda_N P^T$, equivalent to
\begin{equation}\label{eq:perm_L}
    \mathbf{I} =  P^T\widetilde{\Lambda}_NP\mathbf{S} + \mathbf{I}_{\text{in}}.
\end{equation}
As a final remark, we considered the implicit Euler method for the assembly of $\Lambda_N$.
However, a similar derivation can be applied to any formal solver (i.e. ODE integrator), such as high order-exponential integrators or multistep methods.
It would be sufficient to modify~\eqref{ie_mono3}-\eqref{ie_mono4} according to the selected formal solver, while the structure of \eqref{eq:lambda_block} would remain the same.
In terms of spectral analysis, the specific choice of integrator would not change the essence of this work, given that: (i) both transfer and scattering operators can be framed as integral operators, thus compact; (ii) the global radiative transfer operator contains the product $\Lambda_N\Sigma_N$, cf.~(\ref{eq:RT_sys_algebraic}).
The spectral distributions of matrix sequences corresponding to integral operators are clustered at zero, since the continuous operators are compact, with a unique accumulation point at zero. Given the $*$-algebra property of spectral distributions, the same is true for the product matrix sequence $\{\Lambda_N\Sigma_N\}$. Section~\ref{sec:symbol} is devoted to make this intuition formally precise.

\subsection{Scattering operator $\Sigma_N$}\label{sec:S}
To explicitly write the structure of $\Sigma_N$, we can first rewrite the emissivity equation \eqref{eq:emissivity} in terms of the source function
\begin{equation}\label{eq:S}
S(\tau,\mu,\nu)=\gamma(\tau,\mu,\nu)\int_{-1}^{1}\int_F\Phi(\tau,\mu,\mu',\nu,\nu')I(\tau,\mu',\nu')\,\mathrm{d}\nu'\mathrm{d}\mu'+\varepsilon_t(\tau,\mu,\nu)/\chi(\tau,\mu,\nu),
\end{equation}
with $\gamma=\sigma/(2\chi)$. In the context of DOM, we approximate the scattering integral by numerical quadrature:
\begin{align*}    \int_{-1}^{1}\int_F\Phi(\tau,\mu,\mu',\nu,\nu')I(\tau,\mu',\nu')\,\mathrm{d}\nu'\mathrm{d}\mu'& \approx  \sum_{i=1}^{N_\Omega}\sum_{j=1}^{N_\nu}w_{i,j}\Phi(\tau,\mu,\nu,\mu_i,\nu_j)I(\tau,\mu_i,\nu_j)\\
&=\frac{1}{N_r}\sum_{k=1}^{N_r}w_k\Phi(\tau,\bm{r},\bm{r}_k)I(\tau,\bm{r}_k),
\end{align*}
with appropriate quadrature weights $w_k\in\mathbb{R}$. For each spatial node $i=1,\ldots,N_s$, the scattering kernel coefficients can be collected in the (possibly full) matrix $\Psi_{[i]}\in\mathbb{R}^{N_r\times N_r}$:

$$
\Psi_{[i]}=
\begin{bmatrix}
\Phi(\tau_i,\bm{r}_1,\bm{r}_1) & \Phi(\tau_i,\bm{r}_1,\bm{r}_2) & \cdots & \Phi(\tau_i,\bm{r}_1,\bm{r}_{N_r}) \\
\Phi(\tau_i,\bm{r}_2,\bm{r}_1) & \Phi(\tau_i,\bm{r}_2,\bm{r}_2) & \cdots & \Phi(\tau_i,\bm{r}_2,\bm{r}_{N_r}) \\
\vdots  & \ddots & & \vdots\\
\Phi(\tau_i,\bm{r}_{N_r},\bm{r}_1) & \Phi(\tau_i,\bm{r}_{N_r},\bm{r}_2) & \cdots & \Phi(\tau_i,\bm{r}_{N_r},\bm{r}_{N_r})
\end{bmatrix}.
$$
The scattering kernel $\Phi$ is symmetric, i.e. given two distinct rays $\bm{r},\bm{r}'$ then $\Phi(\tau,\bm{r},\bm{r}')=\Phi(\tau,\bm{r}',\bm{r})$ and therefore the matrix $\Psi_i$ is symmetric for all $i$.
Since scattering does not couple different spatial points, the scattering matrix $\Sigma_N$ in equation \eqref{eq:RT_algebraic_form2} is block-diagonal. The block equation corresponding to the optical depth $\tau_i$, for $i=1,\ldots,N_s$, is given by
$$
\mathbf{S}_i=
\Gamma_{[i]}\Psi_{[i]}W_r\mathbf{I}_i+ \mathbf{t}_i,
$$
where $\Gamma_{[i]},W_r\in\mathbb{R}^{N_r\times N_r}$ are the following diagonal matrices
$$
\Gamma_{[i]} =
\begin{bmatrix}
\gamma(\tau_i,\bm{r}_1) &  & & \\
 & \gamma(\tau_i,\bm{r}_2) & & \\
& & \ddots & \\
& & & \gamma(\tau_i,\bm{r}_{N_r})
\end{bmatrix}
\quad \text{and}\quad
W_r = \frac{1}{N_r}
\begin{bmatrix}
w_1 & & & \\
 & w_2 & & \\
& & \ddots & \\
& & & w_{N_r}
\end{bmatrix},
$$
collecting scattering coefficients and quadrature weights respectively. Finally, the total scattering operator $\Sigma_N\in\mathbb{R}^{N\times N}$, can be written as
\begin{equation}\label{eq:Sigma}
\Sigma_N =
\begin{bmatrix}
\Gamma_{[1]} & & & \\
 & \Gamma_{[2]} & & \\
& & \ddots & \\
& & & \Gamma_{[N_s]}
\end{bmatrix}
\begin{bmatrix}
\Psi_{[1]} & & & \\
 & \Psi_{[2]} & & \\
& & \ddots & \\
& & & \Psi_{[N_s]} \\
\end{bmatrix}(I\!d_{N_s}\otimes W_r) = \Gamma_N\Psi_N W_N,
\end{equation}
where both $\Gamma_N,W_N\in\mathbb{R}^{N\times N}$ are diagonal, while $\Sigma_N$ is block diagonal and symmetric.

\subsection{Multi-dimensional case}\label{sec:3D}
%
% \cred{aggiungo figura da 1D a 3D?}
The one dimensional case can be extended to multidimensional setting, considering the domain $D$ to be a cuboid in $d$ dimensions. 
For $d=3$, the ray $\bm{r}=(\mu,\nu)\in[-1,1]\times F$ has to be replaced by the more general $\bm{r}=(\bm\Omega,\nu)=(\mu,\chi,\nu)\in[-1,1]\times[0,2\pi)\times F$.
In terms of transfer, the block diagonal structure of $\widetilde{\Lambda}_N$ is preserved. Moreover, for each spatial point $\bm{x}\in D$ and ray $(\bm{\Omega},\nu)$ it is possible to consider the segment intersecting $\bm{x}$ and some $x_{\min},x_{\max}\in\partial D$ uniquely determined by $\bm{x}$ and $\bm{\Omega}$. As in the 1D case, we can use the corresponding optical depth scale, with an $M$-points\footnote{More precisely, $M=M(\bm{x},\bm{\Omega})$, i.e. the number of interval in the ray partition depends on the spatial point and direction under consideration.} partition $\tau_{\min}=\tau_1<\ldots<\tau_M=\tau_{\max}$ with  $\Delta \tau_i=|\tau_{i+1}-\tau_i|$, for $i=1,\ldots,M-1$, so that \eqref{ie_mono1} and \eqref{ie_mono2} still apply, replacing $N_s$ by $M$.
This technique to propagate radiation in multiple dimensions until reaching the spatial domain boundary is known as a long-characteristic method.
Given a ray $\bm{r}_k$, the corresponding $k$th row in Equation~\eqref{eq:lambda_block} has to be replaced by $L_k$ decoupled vectorial equations, where $L_k$ is the number of parallel directions used to ``cover'' the 3D space\footnote{For example, if the domain $D$ is a cuboid discretized with $N_x\times N_y\times N_z$ points, we can chose $L=N_xN_y$.}:
\begin{equation*}
\widetilde{\mathbf{I}}_{k,\ell} = \Lambda_{[k,\ell]}\widetilde{\mathbf{S}}_{k,\ell}+ (\widetilde{\mathbf{I}}_{\text{in}})_{k,\ell}, \qquad \text{for} \qquad \ell=1,\ldots,L_k,
\end{equation*}
where $\widetilde{\mathbf{I}}_{k,\ell}=[I(\tau_1(k,\ell),\bm{r}_k),I(\tau_2(k,\ell),\bm{r}_k),\ldots,I(\tau_M(k,\ell),\bm{r}_k)]^T\in\mathbb{R}^{M}$, given the partition $\{\tau_m(k,\ell)\}_{m=1}^M$ with $M=M(k,\ell)$.
Since the emissivity will often be computed on a regular grid, not containing the $M$ points where
$\widetilde{\mathbf{I}}_{k,\ell}$ is evaluated, we use interpolation operators to map the nodes $\{\tau_m(k,\ell)\}_{m=1}^M$, for $\ell=1,\ldots,L$, to a set of corresponding nodes in cartesian coordinates $(x,y,z)\in D$:
\begin{equation*}
\widetilde{\mathbf{I}}_k=T_{\mathrm{RC},k}[\widetilde{\mathbf{I}}_{k,1},\widetilde{\mathbf{I}}_{k,2},\ldots,\widetilde{\mathbf{I}}_{k,L}]^T \qquad \text{for} \qquad k=1,\ldots,N_r,
\end{equation*}
where $T_{\mathrm{RC},k}$ is the ``Ray'' to ``Cartesian'' linear interpolation operator (with rows summing up to one) of size $N_s\times N_{k,s}$, where $N_{s,k}=\sum_{\ell=1}^{L_k}M(\ell,k)$ is the total number of spatial points used to compute transfer for the $k$th ray. Similarly,
\begin{align*}
&[\widetilde{\mathbf{S}}_{k,1},\widetilde{\mathbf{S}}_{k,1},\ldots,\widetilde{\mathbf{S}}_{k,L_k}]^T =T_{\mathrm{CR},k}\widetilde{\mathbf{S}}_k, \\
&[(\widetilde{\mathbf{I}}_{\text{in}}^{\uparrow})_{k,1},(\widetilde{\mathbf{I}}_{\text{in}}^{\uparrow})_{k,2},\ldots,(\widetilde{\mathbf{I}}_{\text{in}}^{\uparrow})_{k,L_k}]^T =T_{\mathrm{CR},k}(\widetilde{\mathbf{I}}_{\text{in}}^{\uparrow})_k,
\end{align*}
with $T_{\mathrm{CR},k}\in \mathbb{R}^{N_{k,s}\times N_s}$ is a linear interpolation operator mapping data from the cartesian $(x,y,z)$ grid to the $L_k$ rays. We illustrate an example for a 2D domain in Figure~\ref{fig:grids}. In multiple dimensions, while the notation can become cumbersome, the structure of transfer (cf. Equation~\eqref{eq:lambda_block}) remains the same, with $\Lambda_{[k]}$ collecting $L_k$ independent one-dimensional problems and linear interpolations for each ray index $k=1,\ldots,N_r$:
\begin{equation*}
    \Lambda_{[k]} = T_{\mathrm{CR},k}\Lambda_{\mathrm{R},k} T_{\mathrm{RC},k},
\end{equation*}
with $\Lambda_{\mathrm{R},k}\in\mathbb{R}^{N_{k,s}\times N_{k,s}}$ a block diagonal matrix with $L_k$ blocks simply corresponding to the 1D problems as in \eqref{eq:lamda_down} or \eqref{eq:lamda_up}.
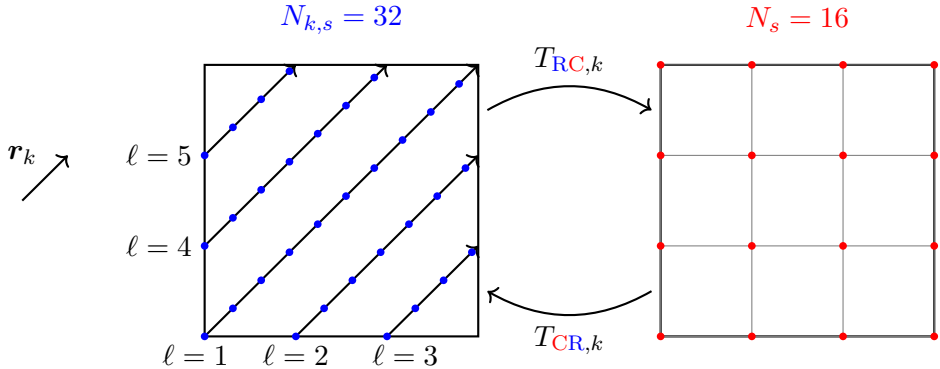
\begin{figure}[ht]
\centering
\begin{tikzpicture}[scale=1.2]

% First square domain (left)
\draw[thick] (0,0) rectangle (3,3);
\node at (1.5,3.5) {\cblue{$N_{k,s}=32$}};

\draw[->, thick] (-2,1.5) -- (-1.5,2);
\node at (-2,2) {$\bm{r}_k$};

\node at (-0.1,-0.2) {$\ell = 1$};
\node at (1,-0.2) {$\ell = 2$};
\node at (2.2,-0.2) {$\ell = 3$};
\node at (-0.5,1) {$\ell = 4$};
\node at (-0.5,2) {$\ell = 5$};

% Draw the arrow
    \draw[->, thick] (0,0) -- (3,3);
    \draw[->, thick] (0,1) -- (2,3);
    \draw[->, thick] (0,2) -- (1,3);
    \draw[->, thick] (1,0) -- (3,2);
    \draw[->, thick] (2,0) -- (3,1);

% Diagonal arrows in the first domain with blue dots
\foreach \x in {-2,...,2} {
    \pgfmathsetmacro{\xstart}{max(0,\x)}
    \pgfmathsetmacro{\ystart}{max(0,-\x)}
    \pgfmathsetmacro{\xend}{min(3,3+\x)}
    \pgfmathsetmacro{\yend}{min(3,3-\x)}

    \foreach \i in {0,...,9} {
        \pgfmathsetmacro{\fx}{\xstart + \i*0.31}
        \pgfmathsetmacro{\fy}{\ystart + \i*0.31}
        \pgfmathparse{\fx>3}
        \ifdim\pgfmathresult pt=1pt
          % fx < 1 --> skip
        \else
          \pgfmathparse{\fy>3}
          \ifdim\pgfmathresult pt=1pt
            % fy < 1 --> skip
          \else
            \fill[blue] (\fx,\fy) circle (1.2pt);
          \fi
        \fi
    }
}

% Second square domain (right)
\draw[thick] (5,0) rectangle (8,3);
\node at (6.5,3.5) {\cred{$N_s=16$}};

% Grid lines in the second domain
\foreach \i in {5,...,8} {
    \draw[gray, thin] (\i,0) -- (\i,3);
}
\foreach \j in {0,...,3} {
    \draw[gray, thin] (5,\j) -- (8,\j);
}

% Red dots at grid intersections including boundaries
\foreach \i in {5,...,8} {
    \foreach \j in {0,...,3} {
        \fill[red] (\i,\j) circle (1.2pt);
        % Optional labels:
        % \node[red, below right] at (\i,\j) {\tiny ($\i$,$\j$)};
    }
}

% Curved arrows (opposite directions)
\draw[->, thick, bend left=30] (3.1,2.5) to node[above] {$T_{\mathrm{\cblue{R}\cred{C}},k}$} (4.9,2.5);
\draw[->, thick, bend left=30] (4.9,0.5) to node[below] {$T_{\mathrm{\cred{C}\cblue{R}},k}$} (3.1,0.5);

\end{tikzpicture}
\caption{Representation of interpolation operators between two discretizations of a square domain $D$ in 2D. In the left ones, (used to compute transfer), give a ray $\bm{r}_k$, we consider $L_k=5$ rays, where the longest one ($\ell=1$) is partitioned with $M(k,1)=10$ nodes. On the right, we show the cartesian grid, where scattering is computed.}\label{fig:grids}
\end{figure}

In terms of the scattering operator $\Sigma_N$, Equation~\eqref{eq:Sigma} still holds. Equation~\eqref{eq:S} is simply replaced by the multidimensional counterpart $S(\bm{x,\bm{\Omega},\nu})=\varepsilon(\bm{x,\bm{\Omega},\nu})/\chi(\bm{x,\bm{\Omega},\nu})$, given \eqref{eq:emissivity}. The derivation leading to Equation~\eqref{eq:Sigma} remains the same.

\subsection{Global operator}
Given the global radiative transfer operator $A_N = I\!d_N-\Lambda_N\Sigma_N$ and $\Lambda_N=P^T\widetilde{\Lambda}_NP$ from \eqref{eq:perm_L} with $P$ a permutation operator, we can write
\begin{align}\label{eq:A_N-1}
    A_N & = I\!d_N-P^T\widetilde{\Lambda}_NP\Sigma_N\\ \label{eq:A_N}
        & = I\!d_N-P^T\widetilde{\Lambda}_NP\Gamma_N\Psi_N W_N,
\end{align}
where we can assume the matrix $\Gamma_N\Psi_N$ to be bounded from above in spectral norm by a constant independent of $N$.
For $d>1$ we have to include interpolation operators between cartesian and ray-wise discretizations in the transfer operator, i.e. $\Lambda_N=P^TT_{\mathrm{CR}}\Lambda_{\mathrm{R}} T_{\mathrm{RC}}P$, with $T_{\mathrm{CR}}=\mathrm{blockDiag}_{k=1}^{N_r}(T_{\mathrm{CR,k}})$ and similarly for $\Lambda_{\mathrm{R}}=\mathrm{blockDiag}_{k=1}^{N_r}(\Lambda_{\mathrm{R},k})$ of size $\sum_{k=1}^{N_r}N_{k,s}$.
In practice, permutations operators are not relevant to the following spectral analysis, since they preserve spectra.
Interpolation operators also do not play a significant role, as detailed in the next section, due to the $*$-algebra property of spectral distributions and because they are bounded in the infinity norm and therefore in the spectral norm (i.e. the Schatten infinity norm, cf. Section~\ref{sec:preliminaries}) given sparsity \cite[Appendix~A]{golinskii2007asymptotic}.

%\cred{@Stefano: Per gli operatori di interpolazione, dovremmo aggiungere qualcosa? E.g. un remark? Per ora sono liquidate senza troppe spiegazioni.}
%\cred{Prima delle 4.1 e 4.2 estendere le intro sulle definizioni e argomenti a seguire?}

\section{Spectral analysis} \label{sec:symbol}

The current section is organized as follows. In Section \ref{sec:preliminaries} we introduce definitions and tools for analyzing $A_N$ and the related matrix sequence $\{A_N\}$, with $A_N$ as in (\ref{eq:A_N-1})-(\ref{eq:A_N}). We are interested in localization results, extremal eigenvalues and singular values, and Weyl distributions: in Section~\ref{sec:results}, after new theoretical results on sparsely undounded matrix sequences i.e. Theorem \ref{th:s.u. *-algebra} and Theorem \ref{th: zero-distr ideal of s.u. *-algebra}, we perform a precise study mainly dealing with Weyl distributions and clustering analysis of our coefficient matrix sequences.

The guiding idea is the following: as clear from (\ref{ie_mono1})-(\ref{ie_mono4}), our matrices stem from the approximation of integral operators; by their compactness, we must expect that the matrix sequences associated with meaningful integration formulae are zero-distributed; see Definition \ref{distribution} with $h\equiv 0$, Definition \ref{def-cluster} with $s=0$, and the results in \cite{al2014singular}.

\subsection{Preliminaries} \label{sec:preliminaries}

As a starting point, we report the essential spectral tools, which are foundational for the theoretical results in Section~\ref{sec:results}, where we describe the spectral behavior and the singular value behavior of the RT matrix sequences, both in term of distributions and clustering analysis.
First, we define the notion of approximating class of sequences and spectral distribution for eigenvalues and singular values, then we report few theoretical tools used to derive the main results. As an intermediate step, we introduce notions and results related to sparsely vanishing and sparsely unbounded functions and matrix sequences. Furthermore, we present the concept of clustering, providing a link to the idea of spectral distributions.

Following the very exhaustive book \cite{bhatia2013matrix}, given a matrix $A_N$ of size $N$ with singular values $\sigma_{1}\geq\sigma_{2}\geq\cdots\geq\sigma_{N}\geq0$, the Schatten $p$ norm is defined as
\begin{eqnarray*}
  &&\left\|A_N\right\|_{S,p}=\left(\sum_{n=1}^{N}\sigma_{n}^{p}\right)^{\frac{1}{p}}, \qquad p\in[1,\infty),\\
  &&\left\|A_N\right\|=\sigma_{1},\qquad p=\infty.
\end{eqnarray*}
We now recall the definition of approximating class of sequences and eigen/singular value distribution, with a useful proposition borrowed from \cite{capizzano2001distribution}.  We then recall the notion of strong and weak clustering \cite{tyrtyshnikov1996unifying}. Finally, Proposition \ref{lem} and Proposition \ref{a.c.s.-product - prop 5.5}, as well as the notions of sparsely vanishing and sparsely unbounded functions and matrix sequences are taken from \cite{GLT-I}.

\begin{definition} \label{appr:seq}
{\rm Suppose a sequence of matrices $\{A_N\}$, $A_N$ of size $d_N$ ($d_k<d_{k+1}$ for each $k\in \mathbb{N}$), is
given. We say that $\{\{B_{N,m}\}:\
m\in\mathbb{N}^+\}$, $B_{N,m}$ of size $d_N$, is an approximating
class of sequences (a.c.s.) for $\{A_N\}$ if, for all
sufficiently large $m\in\mathbb{N}^{+}$, the following splitting
holds:
\begin{equation*}\label{spli1}
A_N=B_{N,m}+R_{N,m}+N_{N,m}\quad \mbox{for all}\, N> n_m,
\end{equation*}
with
\begin{equation*}
\label{propri1} {\rm rank}(R_{N,m}) \leq  d_N\, c(m), \quad \|
N_{N,m}\| \leq  \omega(m),
\end{equation*}
where $n_m$, $ c(m)$ and $\omega(m)$ depend only on $m$ and, moreover,
\begin{equation*}
\label{propri1bis} \lim_{m\to\infty} c(m)=0, \quad
\lim_{m\to\infty} \omega(m)=0.
\end{equation*}}
In short we write $\{\{B_{N,m}\}:\ m\}\stackrel{\rm a.c.s.}{\longrightarrow}\{A_N\}$.
\end{definition}

\begin{definition} \label{distribution}
{\rm We say that a sequence $\{A_N\}$, $A_N$ of size $d_N$ ($d_k<d_{k+1}$ for each $k\in \mathbb{N}$), is distributed in the
eigenvalue sense as a measurable function $h$ (known as eigenvalue or singular value symbol of $\{A_N\}$) over its domain $K$
with positive and finite (Lebesgue) measure $\mu\{K\}$, and we
write $\{A_{N}\}\sim_\lambda (h,K)$, if and only if for every $F\in
{\cal C}_0(\mathbb C)$ (continuous with bounded support over the complex field
$\mathbb C$)
\begin{eqnarray*}
\lim_{N\to\infty} \Sigma_{\lambda}(F,A_{N})= \frac 1 {\mu\{K\}}\int_K
F\bigl(h(x)\bigr)\,\textrm{d}x,
\quad {\rm where} \quad
\Sigma_{\lambda}(F,A_{N})= \frac 1 {d_N}\sum_{i=1}^{d_N}
F\Bigl(\lambda_i^{(N)}\Bigr),
\end{eqnarray*}
with $\left\{\lambda_i^{(N)}\right\}_{i=1}^{d_N}$ denoting the set of
the eigenvalues of $A_N$ (counted with their multiplicities) and
with the integral performed with respect to the same measure
$\mu\{\cdot\}$.

Moreover, the distribution is in the singular value sense,
and we write $\{A_{N}\}\sim_\sigma (h,K)$, if and only if for every
$F\in {\cal C}_0(\mathbb R^{+}_0)$ (continuous with bounded support over the real
nonnegative numbers)
\begin{eqnarray*}
\lim_{N\to\infty} \Sigma_{\sigma}(F,A_{N})= \frac 1
{\mu\{K\}}\int_K F\bigl(|h(x)|\bigr)\,\textrm{d}x, \, {\rm where} \quad
\Sigma_{\sigma}(F,A_{N})= \frac 1 {d_N}\sum_{i=1}^{d_N}
F\Bigl(\sigma_i^{(N)}\Bigr),
\end{eqnarray*}
with $\left\{\sigma_i^{(N)}\right\}_{i=1}^{d_N}$ denoting the set of
the singular values of $A_N$ (counted with their multiplicities)
and with the integral performed with respect to the same measure
$\mu\{\cdot\}$.}
\end{definition}
\begin{proposition}
\label{lem} Suppose a sequence of
matrices $\{A_N\}$, $A_N$ of size $d_N$ ($d_k<d_{k+1}$ for each $k\in \mathbb{N}$),
is given, and $\{\{B_{N,m}\}:\, m\in\mathbb{N}^+\}$, $B_{N,m}$ of size $d_N$, is an a.c.s. for $\{A_N\}$
in the sense of Definition \ref{appr:seq}. Assume  that
$\{B_{N,m}\}\sim_\sigma (h_m,K)$ and that $h_m$ converges in
measure to the measurable function $h$ over $K$, $K$ of finite and
positive measure. Then necessarily
\begin{equation*}
\label{l7} \{A_N\}\sim_\sigma (h,K).
\end{equation*}
Furthermore assume that all $B_{N,m}, N_{N,m}, R_{N,m}$
 are Hermitian and suppose that \\$\{B_{n,m}\}\sim_\lambda (h_m,K)$ with
$h_m$ converging in measure to the measurable function $h$ over
$K$, $K$ of finite and positive measure. Then necessarily
\begin{equation*}
\label{l7-eig} \{A_N\}\sim_\lambda (h,K).
\end{equation*}
\end{proposition}

\begin{definition}[\textbf{sparsely unbounded matrix sequence}]\label{s.u.}
A matrix sequence $\{A_N\}$ is said to be sparsely unbounded \index{Sparsely unbounded (s.u.)\newline matrix-sequence}\index{Matrix-sequence!sparsely unbounded (s.u.)}\index{S.u.}(s.u.), $A_{N}$ of size $d_N$ ($d_k<d_{k+1}$ for each $k\in \mathbb{N}$),
 if for every $M>0$ there exists $n_M$ such that, for $N\ge n_M$,
$$\frac{\#\{i\in\{1,\ldots,d_N\}:\ \sigma_i(A_N)>M\}}{d_N}\le r(M),$$
where $\lim_{M\to\infty}r(M)=0$.
\end{definition}

\begin{definition}[\textbf{sparsely vanishing matrix sequence}]\label{s.v.}
A matrix-sequence $\{A_N\}$ is said to be \index{Sparsely vanishing (s.v.)\newline matrix-sequence}\index{Matrix-sequence!sparsely vanishing (s.v.)}sparsely vanishing \index{S.v.}(s.v.), $A_{N}$ of size $d_N$ ($d_k<d_{k+1}$ for each $k\in \mathbb{N}$), if for every $M>0$ there exists $n_M$ such that, for $N\ge n_M$,
$$\frac{\#\{i\in\{1,\ldots,d_N\}:\ \sigma_i(A_N)<1/M\}}{d_N}\le r(M),$$
where $\lim_{M\to\infty}r(M)=0$.
\end{definition}
It is clear that if $\{A_N\}$ is s.v.\ then $\{A_N^\dag\}$ is s.u.; it suffices to recall that the singular values of the pseudo-inverse $A^\dag$ are $1/\sigma_1(A),\ldots,1/\sigma_r(A),0,\ldots,0$, where $\sigma_1(A)\ldots\sigma_r(A)$ are the nonzero singular values of $A$ ($r={\rm rank}(A)$).

\begin{proposition}\label{a.c.s.-product - prop 5.5}
Let $\{A_N\},\,\{A_N'\}$ be s.u.\ matrix sequences, $A_{N}, A_N'$ of size $d_N$ ($d_k<d_{k+1}$ for each $k\in \mathbb{N}$), and suppose that
\begin{itemize}
	\item $\{\{B_{N,m}\}:\ m\}\stackrel{\rm a.c.s.}{\longrightarrow}\{A_N\}$,
	\item $\{\{B'_{N,m}\}:\ m\}\stackrel{\rm a.c.s.}{\longrightarrow}\{A_N'\}$,
\end{itemize}
both in the sense of Definition \ref{appr:seq}. Then $\{\{B_{N,m}B'_{N,m}\}:\ m\}\stackrel{\rm a.c.s.}{\longrightarrow}\{A_N A_N'\}$.
\end{proposition}
It is worth stressing that the aforementioned distribution results have nice consequences on the cluster structure of the spectrum of the considered matrix sequences. We now report the definition of spectral clustering (of eigenvalues and singular values) and the relations with the spectral distribution notion.
\begin{definition}\label{def-cluster}
{\rm A matrix sequence $\{A_N\}$, $A_{N}$ of size $d_N$ ($d_k<d_{k+1}$ for each $k\in \mathbb{N}$), is {\em properly $($or strongly$)$ clustered at $s \in
\mathbb C$} (in the eigenvalue sense), if for any $\epsilon>0$ the
number of the eigenvalues of $A_N$ off the disk
\begin{eqnarray*}
D(s,\epsilon):=\{z:|z-s|<\epsilon\},
\end{eqnarray*}
can be bounded by a pure constant $q_\epsilon$ possibly depending
on $\epsilon$, but not on $N$. In other words
\begin{eqnarray*}
q_\epsilon(N,s):=\#\{i: \lambda_i^{(N)}\notin D(s,\epsilon)\}=O(1),
\quad N\to\infty.
\end{eqnarray*}
If every $A_N$ has only real eigenvalues (at least for all $N$
large enough), then $s$ is real and the disk $D(s,\epsilon)$
reduces to the interval $(s-\epsilon,s+\epsilon)$. Furthermore,
$\{A_N\}$ is {\em properly $($or strongly$)$ clustered at a
nonempty closed set $S \subset \mathbb C$} (in the eigenvalue
sense) if for any $\epsilon>0$
\begin{equation*} \label{2.2}
q_\epsilon(N,S):=\#\{i: \lambda_i^{(N)}\not\in
D(S,\epsilon)\}=O(1), \quad N\to\infty,
\end{equation*}
$D(S,\epsilon):=\cup_{s\in S} D(s,\epsilon)$ is the
$\epsilon$-neighborhood of $S$, and if every $A_N$ has only real
eigenvalues, then $S$ has to be a nonempty closed subset of
$\mathbb R$. Finally, the term ``properly (or strongly)'' is
replaced by ``weakly'', if
\begin{eqnarray*}
q_\epsilon(N,s)=o(d_N), \qquad \bigl(q_\epsilon(N,S)=o(d_N)\bigr),
\quad N\to\infty,
\end{eqnarray*}
in the case of a point $s$ (a closed set $S$), respectively.
The same notions apply to singular values: in that case $s$ has to
be a real nonnegative number and $S$ a subset of the real
nonnegative numbers.}
\end{definition}
\begin{remark}\label{clusters vs distributions}
{\rm To link the concept of cluster with the distribution notions, it is
instructive to observe that $\{A_N\}\sim_{\lambda} (h,K)$, with
$h\equiv c$ being a constant function, is equivalent to write that
$\{A_N\}$ is weakly clustered at $c \in \mathbb C$ in the sense of the eigenvalues.
Analogously, $\{A_N\}\sim_{\sigma} (h,K)$, with
$h\equiv \hat c$ being a nonnegative constant function, is equivalent to write that
$\{A_N\}$ is weakly clustered at $\hat c \in {\mathbb R}^+_0$ in the sense of the singular values.
Furthermore, by looking at Definition \ref{distribution}, if we write $\{A_N\}\sim_{\lambda} (c,K)$, $c \in \mathbb C$, or $\{A_N\}\sim_{\sigma} (\hat c,K)$, $\hat c \in {\mathbb R}^+_0$, respectively, then it is true that $\{A_N\}\sim_{\lambda} (c,\tilde K)$, $c \in \mathbb C$, or $\{A_N\}\sim_{\sigma} (\hat c,\tilde K)$, $\hat c \in {\mathbb R}^+_0$, for every measurable $\tilde K$ with finite and positive measure $\mu\{\tilde K\}$, as required by Definition \ref{distribution}.
In other words, in the case of a constant distribution function, we can simplify the notation by dropping the indication of the domain since any admissible domain is allowed: hence we write $\{A_N\}\sim_{\lambda} c$ instead of $\{A_N\}\sim_{\lambda} (c,K)$ or $\{A_N\}\sim_{\sigma} \hat c$ instead of
$\{A_N\}\sim_{\sigma} (\hat c,K)$.}
\end{remark}

On the other hand, if $h$ is a generic measurable function and
$\{A_N\}\sim_{\lambda} (h,K)$, then the essential range\footnote{ Given a measurable complex-valued function $h$ defined on a
Lebesgue measurable set $K$, the {\em essential range of $h$} is
the set ${S}(h)$ of points $s\in \mathbb C$ such that, for every
$\varepsilon>0$, the Lebesgue measure of the set
$h^{(-1)}(D(s,\varepsilon)):=\{t\in K:\, h(t)\in
D(s,\varepsilon)\}$ is positive \cite{rudin1987real}.} of $h$,
denoted by $S(h)$, is a weak cluster for the spectra of the matrix
sequence $\{A_N\}$.
In order to have strong clustering we need to apply Proposition \ref{lem},
where concerning the rank correction matrix $R_{n,m}$ in Definition \ref{appr:seq} we have
rank$(R_{N,m})\leq c(m)$  instead of rank$(R_{N,m})\leq d_{N}c(m)$.
Furthermore, it should be noticed that that the notion of strong clustering at zero of
a matrix sequence $\{A_N\}$ is intimately related to the notion of collective compactness
\cite{anselone1971collectively} of  $\{{\mathcal L}_N\}$, with ${\mathcal L}_N$ the linear operator associated with the matrix $A_N$.

The following results, proven in \cite{serra2005deduce} for part 1 and then refined in \cite{al2014singular} for part 2, are the crucial theoretical tool used to describe the spectral properties of the scattering operator $\Sigma_N$.

\begin{theorem}\label{lem1} Let $\{A_{N}\}$ be a sequence of matrices, $A_N$ of size $d_N$ ($d_k<d_{k+1}$ for each $k\in \mathbb{N}$). The following two statements are true.
\begin{enumerate}
\item Assume that there exists $M>0$, independent of $N$
such that $\|A_{N}\|_{S,p}\leq M$,  for a certain norm Schatten $p$, $1\le p<\infty$. Then the sequence $\{A_{N}\}$ has a strong cluster at zero, both in the sense of singular values and eigenvalues.
\item If $\|A_{N}\|_{S,p}^p =o(d_N) $, for a certain norm Schatten $p$, $p<\infty$, then the sequence $\{A_{N}\}$ has a weak cluster at zero in the sense of singular values; if, moreover, the sequence $\{A_{N}\}$ is uniformly bounded in spectral norm, then it is also weakly clustered at zero in the eigenvalue sense.
\end{enumerate}
\end{theorem}
Assume that $k$ is a continuous bi-variate function, $k(\cdot,\cdot):\Omega\times\Omega\rightarrow\mathbb{R}$,
$\Omega\subset\mathbb{R}^{d}$, $d\geq1$, and suppose that
\begin{eqnarray}\label{inte}
  \int_{\Omega}k(x,y)f(y)\,\textrm{d}y=l(x),\qquad x\in\Omega.
\end{eqnarray}
If we consider $\Omega=[0,1]$ and we apply the rectangle rule for approximating the integral
(\ref{inte}), with step $H=1/N$, we have
\begin{eqnarray*}
  l(iH)=H\sum_{j=0}^{N-1}k(iH,jh)f(jH),\qquad i=0,\ldots,N-1.
\end{eqnarray*}
Let us define the matrix $A_{N}$ as
\begin{eqnarray}\label{an}
  A_{N}=H\left[k(iH,jH)\right]_{i,j=0}^{N-1}.
\end{eqnarray}
Such a matrix $A_{N}$ is the sampling matrix of the function $k$ scaled by $H$; clearly,
by varying $N$, that is the finesse parameter, we have a sequence $\{A_{N}\}$ of scaled sampling matrices.

\begin{theorem}\label{th-cluster} The sequence $\{A_{N}\}$, $A_{N}$ defined as in (\ref{an}),
shows a strong cluster at zero both in the singular value and eigenvalue sense, $\forall k\in {\cal C}(\Omega\times\Omega)$.
\end{theorem}

\begin{remark}\label{general Omega and k}
{\rm While Theorem~\ref{th-cluster} is based on the use of the rectangular quadrature rule over the unit segment, the proof can be generalized and holds for $k(\cdot,\cdot)$ being only Riemann integrable, and the domain $\Omega$ being a general Peano–Jordan set in $\mathbb{R}^d$, and if the quadrature rule is different (see Section~4 in~\cite{al2014singular}).}
\end{remark}
\begin{remark}\label{comput}
{\rm The latter result has an interesting computational application, for example in terms of preconditioning. The use of the singular value decomposition, together with the result in Theorem \ref{th-cluster}, implies that for every positive $\epsilon$, there exist a matrix $A_{N,\epsilon}$ and a value $r_\epsilon$ such that rank$(A_{N,\epsilon})\le r_\epsilon$ and $\|A_N-A_{N,\epsilon}\|\le \epsilon$. In other words, the action of $A_N$ can be approximated, with $\epsilon$ error in  Euclidean norm, by a fast procedure having linear cost with respect to the size $N$ and with constant proportional to $r_\epsilon$ (see also \cite{reichel1990fast}).}
\end{remark}

\subsection{Theoretical results}\label{sec:results}
The key takeaway of this section is that the global coefficient matrix $A_N$ is a good asymptotic approximation in clustering terms and in the a.c.s. topology \cite{GLT-I} of the identity matrix, despite the dense nature of the large diagonal blocks. Given the tools reported in
the previous section, this is not completely a surprise since these blocks originate from the numerical approximation of compact integral operators.
This contribution provides a clear understanding of the reason why solving RT systems with coefficient matrix $A_N$ is robust when using Krylov methods, even without preconditioning and even for extremely large matrix-sizes.

From a technical point of view we first give general results on sparsely unbounded matrix sequences taken from \cite{GLT-III}, on their algebraic structure, and on their relationships with zero-distributed matrix sequence. Then we proceed with the analysis of RT coefficient matrices.

\begin{proposition}\label{s.u.pro}
Let $\{A_N\}$ be a matrix sequence. The following conditions are equivalent.
\begin{enumerate}
	\item $\{A_N\}$ is s.u.
	\item $\displaystyle\lim_{M\to\infty}\limsup_{N\to\infty}\frac{\#\{i\in\{1,\ldots,d_N\}:\ \sigma_i(A_N)>M\}}{d_N}=0$.
	\item For every $M>0$ there exists $n_M$ such that, for $N\ge n_M$,
	\begin{equation*}
	A_N=\hat A_{N,M}+\tilde A_{N,M},\qquad{\rm rank}(\hat A_{N,M})\le r(M)d_N,\qquad \|\tilde A_{N,M}\|\le M,
	\end{equation*}
	where $\lim_{M\to\infty}r(M)=0$.
\end{enumerate}
\end{proposition}

\begin{proposition}\label{d->su - prop 5.4}
Assume $(h,K)$ from Definition \ref{distribution} exists: if $\{A_N\}\sim_\sigma (h,K)$ then $\{A_N\}$ is s.u. 
\end{proposition}
\begin{proposition}\label{d0->sv - prop 8.4}
Assume $(h,K)$ from Definition \ref{distribution} exists: if $\{A_N\}\sim_\sigma (h,K)$ then $\{A_N\}$ is s.v.\ if and only if $h\ne0$ a.e. on $K$.
\end{proposition}

%\begin{proposition}\label{s.u.p - prop 2.19}If $\{A_N\},\{A_N'\}$ are s.u.\ then $\{A_N A_N'\}$ is s.u.
%\end{proposition}

\begin{theorem}\label{th:s.u. *-algebra}
Let $\{d_N\}$ be a sequence of positive integers with $d_k<d_{k+1}$ for each $k\in \mathbb{N}$ and let ${\mathcal M}_{k}$ be the the space of square complex matrices of size $k$.
   Let ${\mathcal A}=\left\{\{A_N\}_N, A_N\in {\mathcal M}_{d_N}\right\}$ be the $*$-algebra of matrix sequences and let
\[
{\mathcal SU}=\left\{\{A_N\}_N, A_N\in {\mathcal M}_{d_N},\ \{A_N\}_N\ s.u. \right\}.
\]
Then ${\mathcal SU}\subset {\mathcal A}$ is a $*$-algebra as well.
\end{theorem}
{\bf Proof:}\
the claimed $*$-algebra thesis amounts in the following two items. If $\{A_N\},\{A_N'\}$ are s.u.\ and $\alpha, \beta$ are given complex constants then
\begin{enumerate}
\item $\{A_N A_N'\}$ is s.u.,
\item $\{\alpha A_N + \beta A_N'\}$ is s.u.
\end{enumerate}
For the closure of s.u. matrix sequences under multiplication, i.e. item 1, use \cite[Proposition 2.18, p. 50]{GLT-III}, whose proof can be found in \cite{GLT-I}; see also \cite[Propositions 2.19 and 2.20, p.50]{GLT-III}, whose proofs are given in detail in \cite{garoni2019block}. In terms of the tools introduced here, it is enough to invoke Proposition \ref{a.c.s.-product - prop 5.5} and the equivalence between part 1 and part 3 of Proposition \ref{s.u.pro}.

For the linear combination, the property follows from the characterization of s.u. matrix sequence as a sum of low-rank plus norm-norm terms (see the third part of Proposition \ref{s.u.pro}, in the spirit of the a.c.s. topology) plus the sublinear character of norm and rank.
More precisely, $\{A_N\},\{A_N'\}$ s.u.\ imply $\{\alpha A_N\},\{\beta A_N'\}$ s.u. In fact, by Definition \ref{s.u.}, for $\alpha \neq 0$ and for proving that $\{\alpha A_N\}$ is s.u., it is enough to replace the term $M$ with the term $|\alpha| M$, which can be done safely since $\alpha$ is a fixed nonzero constant. When $\alpha=0$ there is nothing to prove since a zero matrix sequence is of course uniformly bounded, and hence s.u. In the same way, we deduce that $\{\beta A_N'\}$ is s.u.

Hence, the proof of item 2 is reduced to show that a sum of two s.u.\ matrix sequences $\{X_N\},\{Y_N'\}$ is still s.u., taking $X_N=\alpha A_N$ and $Y_N=\beta A_N'$.
By item 3 in Proposition \ref{s.u.pro}, we deduce that for every $M>0$ there exists $n_M$ such that, for $N\ge n_M$,
	\begin{eqnarray*}
	X_N & = & \hat X_{N,M}+\tilde X_{N,M},\qquad{\rm rank}(\hat X_{N,M})\le r_X(M)d_N,\qquad \|\tilde X_{N,M}\|\le M, \\
Y_N & = & \hat Y_{N,M}+\tilde Y_{N,M},\qquad{\rm rank}(\hat Y_{N,M})\le r_Y(M)d_N,\qquad \|\tilde Y_{N,M}\|\le M,
	\end{eqnarray*}
	where $\lim_{M\to\infty}r_X(M)=\lim_{M\to\infty}r_Y(M)=0$.
Now we observe that both rank and norm are sublinear and hence
\[
Z_N=X_N+Y_N = \hat X_{N,M/2}+\hat Y_{N,M/2} + \tilde X_{N,M/2}+\tilde X_{N,M/2},
\]
with rank$(\hat X_{N,M/2})\le r_X(M/2)d_N$, rank$(\hat Y_{N,M/2})\le r_Y(M/2)d_N$, $\|\tilde X_{N,M/2}\|\le M/2$, $\|\tilde Y_{N,M/2}\|\le M/2$.
Therefore
	\begin{eqnarray*}
	Z_N & = & \hat Z_{N,M}+\tilde Z_{N,M},\qquad{\rm rank}(\hat Z_{N,M})\le (r_X(M/2)+r_Y(M/2))d_N,\qquad \|\tilde Z_{N,M}\|\le M,
	\end{eqnarray*}
with $\hat Z_{N,M}=\hat X_{N,M/2}+\hat Y_{N,M/2}$, $\tilde Z_{N,M}=\tilde X_{N,M/2}+\tilde Y_{N,M/2}$, $r_Z(M)=r_X(M/2)+r_Y(M/2)$ and
	where $\lim_{M\to\infty}r_Z(M)=0$.
In conclusion, again by virtue of item 3 in Proposition \ref{s.u.pro}, we infer $\{Z_N\}$  s.u.\ with $Z_N=\alpha A_N+ \beta A_N'$.
\hfill $\bullet$
\ \\
The following result has been essentially proved in  \cite[Lemma 3.5]{capizzano2003analysis}, using the tools available at that time. Below we give an a.c.s. based proof.

\begin{theorem}\label{th: zero-distr ideal of s.u. *-algebra}
Let ${\mathcal A}$ and ${\mathcal SU}$ as in Theorem \ref{th:s.u. *-algebra}. \\Let ${\mathcal O}=\left\{\{A_N\}_N, A_N\in {\mathcal M}_{d_N},\ \{A_N\}_N
\sim_\sigma 0 \right\}$. Then  ${\mathcal O}$ is a two-sided ideal of ${\mathcal SU}$.
\end{theorem}
{\bf Proof:}\  We have to prove that a s.u. matrix sequence multiplied either on the left or on the right zero-distributed matrix sequence is still zero-distributed.
Use Propositions 5.4—5.5 in the book \cite{GLT-I} and the fact that the constant a.c.s. $\{\{A_N\}: m\}$ is a.c.s. for the zero matrix sequence $\{O_N\}$ if and only if $\{A_N\}$ is zero-distributed (and this fact follows by virtue of the equivalence below Equation (5.6) at p. 70 in \cite{GLT-I}).
\hfill $\bullet$
\ \\

Now, after the general results, we study our coefficient RT matrices $\{A_N\}$.
\begin{theorem}\label{th:main1}
    Let the coefficient $\gamma=\sigma/2\chi$ from Equation~\eqref{eq:S} be Riemann integrable in general sense (possibly undounded), the domain $\mathcal{D}$ a Peano–Jordan measurable set in $\mathbb{R}^{2d}$, and the scattering kernel $\Phi(\bm{x},\bm{r},\bm{r}')$ be Riemann integrable for any $\bm{x}\in \mathcal{D}$. Then, the sequence of scattering operators $\{\Sigma_N\}$, defined in \eqref{eq:Sigma}, shows a weak cluster at zero in the singular value sense, i.e.
$$\{\Sigma_N\}\sim_\sigma (0,K)$$
for $K$ as in Definition~\ref{distribution}.
\end{theorem}
{{\bf Proof:}\ by definition we have $\Sigma_N=\Gamma_N\Psi_NW_N$, where $\Gamma_N$ is simply the diagonal matrix sampling the function $\gamma$.
Now by Axiom {\bf GLT 3}, part 2, in \cite[p. 170]{GLT-I} we have $\{\Gamma_N\}\sim_{\rm GLT} (\gamma,K)$ and by Axiom {\bf GLT 1} in \cite[p. 170]{GLT-I} we infer $\{\Gamma_N\}\sim_{\sigma} (\gamma,K)$ so that $\{\Gamma_N\}$ is s.u.\ by Proposition \ref{d->su - prop 5.4}.
Now, $\Psi_NW_N$ discretizes the integral operator containing the kernel $\Phi$, as in \eqref{an}, and hence $\{\Psi_NW_N\}\sim_\sigma (0,K)$, which is equivalent to write
$\{\Psi_NW_N\}\in {\mathcal O}$. Since ${\mathcal O}$ is a two-sided ideal of ${\mathcal SU}$ by Theorem \ref{th: zero-distr ideal of s.u. *-algebra} we deduce that
\[
\{\Sigma_N\}\sim_{\sigma} (0,K)
\]
that is $\{\Sigma_N\}\sim_{\sigma} 0$ which is equivalent to the weak clustering at zero.
%Now use Theorem~\ref{th-cluster} for concluding for the eigenvalues.
\hfill $\bullet$
}

\begin{theorem}\label{th:main2}
    Let the coefficient $\gamma=\sigma/2\chi$ be bounded and Riemann integrable, the domain $\mathcal{D}$ a Peano–Jordan measurable set in $\mathbb{R}^{2d}$, and the scattering kernel $\Phi(\bm{x},\bm{r},\bm{r}')$ be Riemann integrable for any $\bm{x}\in K$. Then, the sequence of scattering operators $\{\Sigma_N\}$, defined in \eqref{eq:Sigma}, shows a strong cluster at zero both in the singular value and eigenvalue sense, i.e.
$$\{\Sigma_N\}\sim_\sigma (0,K),\qquad  \{\Sigma_N\}\sim_\lambda(0,K),$$
for $K$ as in Definition~\ref{distribution}.
\end{theorem}
{{\bf Proof:}\ in the case of boundedness of the function $\gamma=\sigma/2\chi$, similarly to Theorem \ref{th-cluster}, one can prove that $\{\Sigma_N\}$ is collectively bounded so that $\|\Sigma_N\|_{S,\infty} =O(1)$,  as $N\rightarrow \infty$. Furthermore, a direct computation shows that $\|\Sigma_N\|_{S,2} =O(1)$ as $N\rightarrow \infty$. As a consequence, by invoking part 1 of Theorem \ref{lem1}, we deduce the strong clustering of the both singular values and eigenvalues.
\hfill $\bullet$
}

\begin{remark}\label{comput3}
{\rm If we remove the assumption of boundedness of the function $\gamma=\sigma/2\chi$ we do not maintain in general the strong clustering while the weak clustering and distribution still stand, but only in the singular value sense, as long as the function $\gamma$ is Riemann integrable in any compact set of the Peano–Jordan measurable set $K$. However, for the weak clustering of the eigenvalues one can use part 2 of Theorem \ref{lem1}; see also the refined analysis in \cite{al2014singular}.}
%$\mathcal{D}$}.
\end{remark}

\begin{theorem}\label{th:main3}
    The sequence of transfer operators $\{\Lambda_N\}_N$, defined in \eqref{eq:perm_L}, is zero-distributed with a strong singular value and eigenvalue clustering.
\end{theorem}
{{\bf Proof:}\  we first observe that $\Lambda_N$ is block-diagonal  (cf. Equation~\eqref{eq:lambda_block}), so that
$\{\Lambda_N\}_N$ is zero-distributed if and only if each matrix-sequence block is zero-distributed. We can use the explicit expression of the entries of each block, which are in the form $\Delta\tau_i/(1+\Delta\tau_i)^{i-j+1}$ according to \eqref{eq:perm_L}. Since this formula is derived from the approximation of a one-dimensional integral operator, Theorem \ref{th-cluster} allows to conclude the proof.
\hfill $\bullet$
}

\begin{remark}\label{comput2}
{\rm Notice that Theorem \ref{th:main3} can be alternatively proven using part 1 of Theorem \ref{lem1} with $p=2$, thanks to the explicit expression $\Delta\tau_i/(1+\Delta\tau_i)^{i-j+1}$; see \eqref{eq:perm_L}.}
\end{remark}

\begin{theorem}\label{th:main4}
    The sequence of the radiative transfer operators $\{A_N\}_N=\{I\!d_N-\Lambda_N\Sigma_N\}_N$, c.f. \eqref{eq:A_N}, is distributed as the identity. More precisely, we have
    $$\{A_N-I\!d_N\}\sim_\sigma 0,\qquad  \{A_N-I\!d_N\}\sim_\lambda 0,$$
which means that the weak clusterings hold. Furthermore, under the assumption that the function $\gamma=\sigma/2\chi$ is bounded,  we deduce a strong eigenvalue cluster and singular value at $1$.
\end{theorem}
{\bf Proof:}\  using the previous results, we have $\{\Sigma_N\}\sim_\sigma 0$ and $\{\Lambda_N\}_N\sim_\sigma 0$ so that
$\{\Lambda_N\Sigma_N\}, \ \{A_N-I\!d_N\}\sim_{\sigma} 0$, because $A_N-I\!d_N=-\Lambda_N\Sigma_N$ and using Theorem \ref{th: zero-distr ideal of s.u. *-algebra}.
The second part regarding the strong eigenvalue cluster and singular value at $1$ is a plain consequence of Theorem \ref{th:main2} and of Theorem
\ref{th:main3}.
\hfill $\bullet$

%[Si potrebbe anche menzionare (magari in un corollario) le ipotesi su $\Lambda$ che renderebbero il teorema vero.]}

\section{Numerical experiments}\label{sec:exp}
For the experimental analysis, we start considering the plane-parallel, cylindrically-symmetric model in \eqref{RT_tau}, since computing eigenvalues is challenging for full 3D models \cite{stamnes1988computation}. We present a more realistic 3D experimental setup in Section~\ref{sec:3d}. We consider various assumptions in the description of scattering processes, but we remark that the theoretical results hold also in the case of the most general scattering kernel, i.e. for dipole scattering with angle-dependent PRD effects \cite{riva2025,guerreiro2024modeling}.

For the discrete spatial grid, we consider an equidistant partition of $[t_{\mathrm{surf}},t_{\mathrm{deep}}]=[0,1]$, with $\Delta t=1/(N_s-1)$, where $t$ is the frequency-integrated optical depth scale along the vertical. %, i.e., ${\mathrm d}t = −[\kappa(z) + \sigma(z)]{\mathrm d}z$}.
The optical depth for a given direction and frequency is therefore: $\tau(\mu,\nu)=\phi(\nu)t/\mu$, where we assume that $\phi$ does not vary in space and direction.
As a common setting for all 1D experiments, we set the following boundary conditions in \eqref{RT_tau}: $I_\mathrm{in}(t_{\mathrm{deep}},\nu) = 1$ and $I_\mathrm{in}(t_{\mathrm{surf}},\nu) = 0$, for all $\nu\in F$. %and the emissivity source term\cluca{[!]} to zero, i.e. $\varepsilon_t=0$ in \eqref{eq:emissivity}. \cluca{Consistenty with this last assumption, we also set $\kappa=0$.} 
%Both these choices do 
This choice does not impact the numerical tests of radiative transfer operators spectra, since boundary conditions and source terms appear only on the right-hand side of the discrete radiative transfer problem, cf. Equation \eqref{eq:RT_sys_algebraic}. 
%
%For the discrete spatial grid, we consider an equidistant partition of $[\tau_{\mathrm{surf}},\tau_{\mathrm{deep}}]=[0,1]$, with $\Delta\tau=1/(N_s-1)$. 
For the angular discretization, we use two sets of Gauss-Legendre nodes and weights, $N_\Omega/2$ for the interval $[-1,0)$ and $N_\Omega/2$ for $(0,1]$. In terms of frequency discretization, when present, we consider the reduced frequencies\footnote{A reduced frequency is non-dimensional distance from a reference value (line-center), normalized to a factor proportional to the width of the absorption profile. Reduced frequencies are usually denoted with the letter $u$. To avoid complicating notation, we use the letter $\nu$ also for reduced frequencies.} interval $F=[-10,10]$, and perform numerical integration using the trapezoidal rule.

In the following sections, we test numerically the eigen-distribution of radiative transfer operators and the convergence iterative solvers, for different scattering kernels, in view of the theoretical results presented in Section~\ref{sec:symbol}.
For the iterative solution of 1D radiative transfer systems, we use both GMRES (with no restart) and BiCGStab methods and
 the constant right hand side $\mathbf{b}=\mathbf{1}$, without explicitly setting the source term $\varepsilon_t$.

\subsection{Monochromatic problem}
First, we consider the monochromatic version of \eqref{eq:RT}, with the frequency domain being a single frequency, %i.e. $F=\{\nu_0\}$, 
and a frequency-independent scattering kernel depending only on the incoming and outgoing directions, i.e.
\begin{equation*}
    \Phi(\mu,\mu')=\sum_{\ell=0}^Ld_\ell P_\ell(\mu)P_\ell(\mu'),
\end{equation*}
where $P_\ell(\mu)$ is the Legendre polynomial of order $\ell$. 
We note that the corresponding discrete operator $\Psi_k$ will have rank $L+1$ for any $k=1,\ldots,N_s$.
We consider the following coefficients with $L=7$ from \cite{pontaza2005least}:
$\{d_0,d_1,\ldots,d_7\}=\{1,1.98398,1.50823,0.70075,\\0.23489,0.05133,0.00760,0.00048\}.$
With this assumption, since different frequencies are decoupled, we can replace \eqref{RT_tau} by its monocromatic version:
\begin{equation}\label{eq:mono_sol}
\frac{\mathrm{d}}{\mathrm{d}\tau}
	I(\tau,\mu)  = I(\tau,\mu) - \frac{1-\mu\tau}{2}\int_{-1}^1\Phi(\mu,\mu')I(\tau,\mu')\,\mathrm{d}\mu',
\end{equation}
where we have reported $S$ explicitly, set $\phi=1$, and $\gamma(t)=(1-t)/2=(1-\mu\tau)/2$, c.f. Equation~\eqref{eq:S}. 
An example of the solution of \eqref{eq:mono_sol} is given in Figure~\ref{fig:mono}.

\begin{figure}
    \centering
\includegraphics[width=\linewidth]{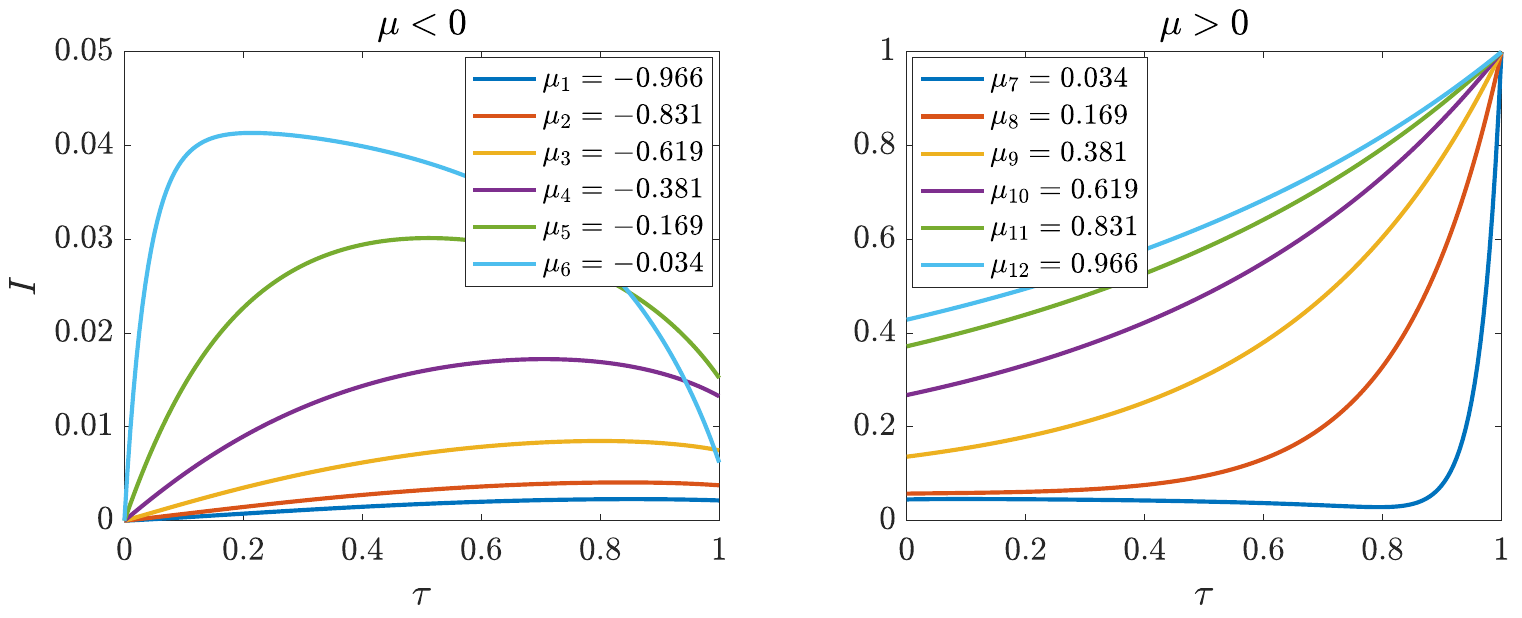}
    \caption{Monochromatic solution of \eqref{eq:mono_sol} discretized with $N_s=200$ and $N_\Omega=N_r=12$, with a discontinuity for $\mu=0$.}
    \label{fig:mono}
\end{figure}

Spectrum inspection reveals a strong clustering of eigenvalues around one. In Table~\ref{tab:mono}, we show the variation in the number of eigenvalues with absolute value in the interval [0.999,1] w.r.t. the discretization parameters. As expected by the theory, eigenvalues accumulate at one under refinement. We remark that non-zero (even if relatively small) imaginary parts are present.
We further note that $|\lambda|_{\min}>0.82$ in all cases, signaling a lower bound in the spectrum and the absence of outliers possibly close to zero.  As expected, these spectral properties correspond to fast and robust convergence of Krylov methods, even without preconditioning, cf. Fig.~\ref{fig:mono_conv}.

\begin{figure}[]
  \centering
  \begin{minipage}[c]{0.5\textwidth}
    \centering
    \scalebox{0.8}{
    \begin{tabular}{cc|cccc}
%\toprule
& & \multicolumn{4}{c}{$N_\Omega$} \\
& & 12 & 24 & 50 & 100\\
\midrule
\multirow{6}{*}{$N_s$}
& 10  & 68.3\% & 84.2\% & 91.8\% & 96.2\% \\
& 20  & 69.6\% & 84.8\% & 92.3\% & 96.4\% \\
& 40  & 73.5\% & 87.8\% & 93.1\% & 96.8\% \\
& 100 & 79.8\% & 90.1\% & 94.1\% & 97.6\% \\
& 200 & 86.6\% & 93.8\% & 95.4\% & 98.5\%
%\bottomrule
\end{tabular}
}
    \captionof{table}{Percentage of eigenvalues with absolute value in the interval [0.999,1] for the monochromatic problem.}
    \label{tab:mono}
  \end{minipage}
  \hfill
  \begin{minipage}[c]{0.48\textwidth}
    \centering
    \includegraphics[width=\textwidth]{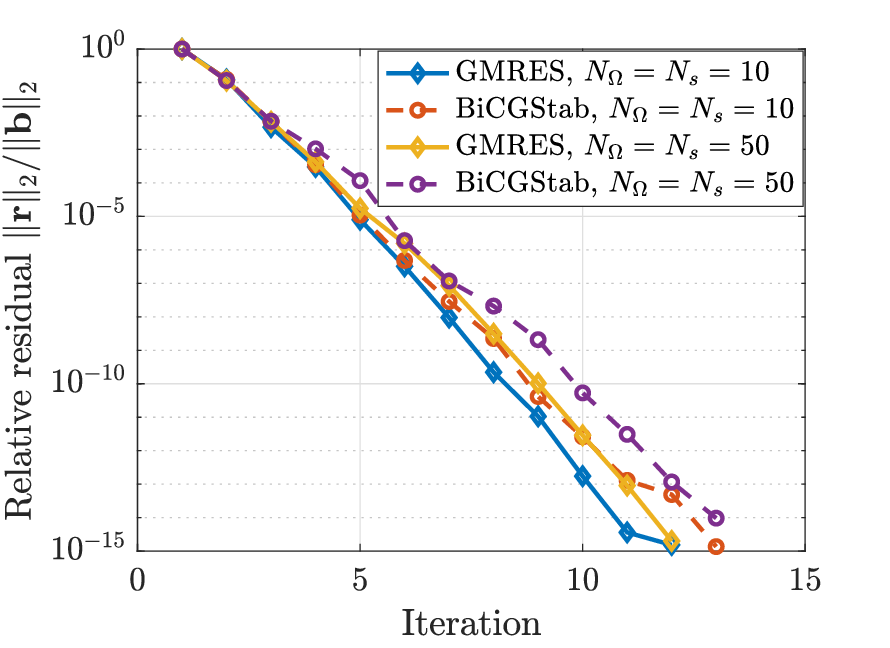}
    \caption{GMRES and BiCGStab convergence for different discretization parameters for the monochromatic problem.}
    \label{fig:mono_conv}
  \end{minipage}%
\end{figure}

\subsection{Isotropic problem}
Secondly, we consider the case of isotropic scattering, i.e. in the form 
\begin{equation}    S(\tau,\mu,\nu)=\gamma(\tau,\mu,\nu)\int_F\Phi(\nu,\nu')\int_{-1}^1I(\tau,\mu',\nu')\,\mathrm{d}\mu'\mathrm{d}\nu',
\end{equation}
In the case of coherent scattering, we have $\Phi(\nu,\nu')=\phi(\nu')\delta(\nu-\nu')$ with $\phi$ a normalized absorption profile, e.g. a Voigt function, obtaining an source function in the form
\begin{equation*}
S_{\mathrm{COH}}(\tau,\mu,\nu)=\gamma(\tau,\mu,\nu)\phi(\nu)\int_{-1}^1I(\tau,\mu',\nu)\,\mathrm{d}\mu'.
\end{equation*}
For the complete redistribution (CRD) case, we have instead $\Phi(\nu,\nu')=\phi(\nu')\phi(\nu)$, obtaining
\begin{equation}
S_{\mathrm{CRD}}(\tau,\nu)=\gamma(\tau,\mu,\nu)\phi(\nu)\int_F\phi(\nu')\int_{-1}^1I(\tau,\mu',\nu')\,\mathrm{d}\mu'\mathrm{d}\nu',
\end{equation}
where photons are redistributed across the line profile.
For the numerical experiments, we set the coefficient

$$\gamma(\tau,\mu,\nu)\phi(\nu)=\frac{1}{2}(1 - \tau\mu/\phi(\nu))=\frac{1}{2}(1-t)=\frac{\sigma(t)\phi(\nu)}{2\chi(t,\nu)},$$ 
corresponding to $\sigma(t)/(\kappa(t)+\sigma(t))=1-t$, and use the following absorption profile 
\begin{equation*}
    \phi(\nu)=\frac{1}{\pi}\frac{1}{\nu^2+1}.
\end{equation*}
We refer to Figures~\ref{fig:iso_sol}-\ref{fig:CRD_sol} for an illustration of the solution in the case of isotropic scattering, in the case of coherent and CRD scattering respectively.
\begin{figure}
    \centering
    \includegraphics[width=\linewidth]{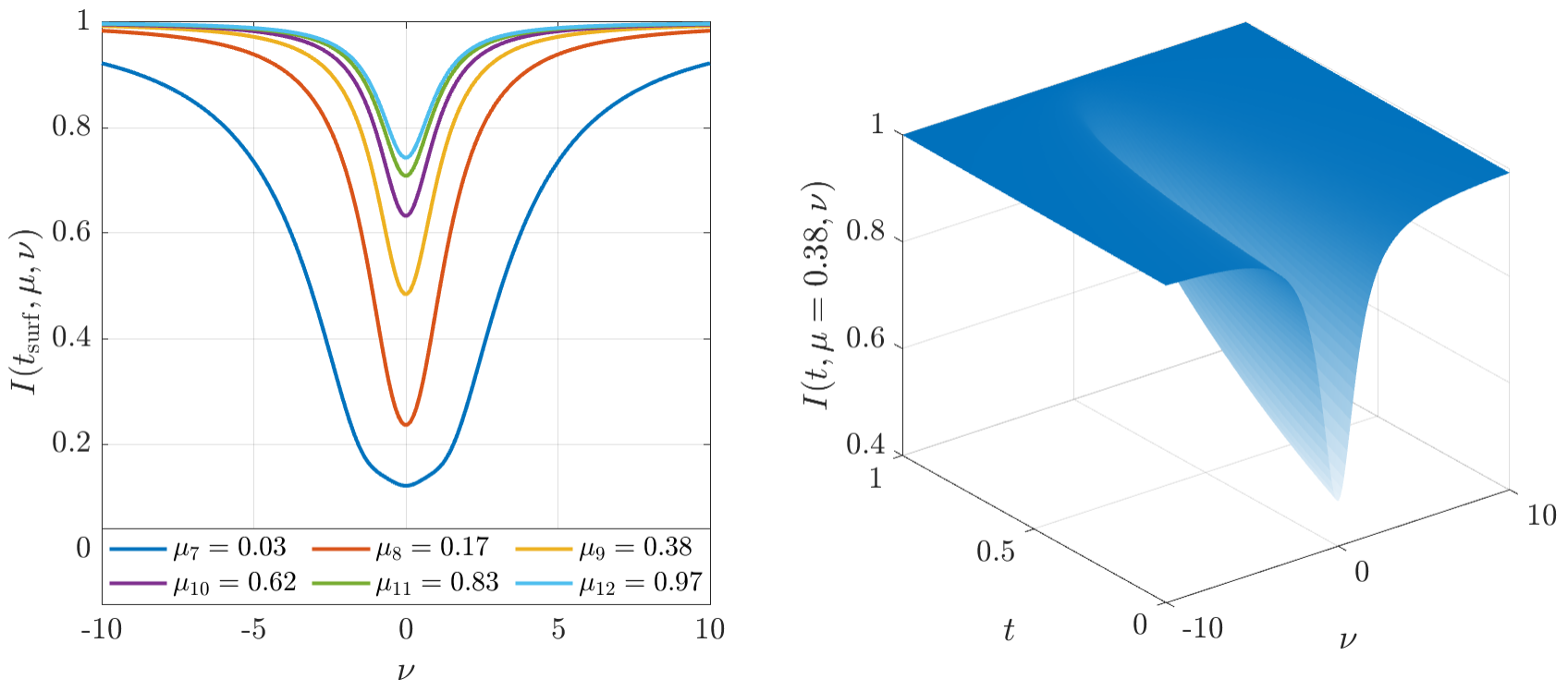}
    \caption{Isotropic solution using coherent scattering, obtained with $N_s=N_{\nu}=200$ and $N_\Omega=12$. Left: emerging radiation profiles at $t_{\mathrm{surf}}=\tau_{\mathrm{surf}}=0$ for various directions $\mu$. Right: Solution profile for $\mu=0.38$ in the full spatial domain.}
    \label{fig:iso_sol}
\end{figure}
\begin{figure}
    \centering
    \includegraphics[width=\linewidth]{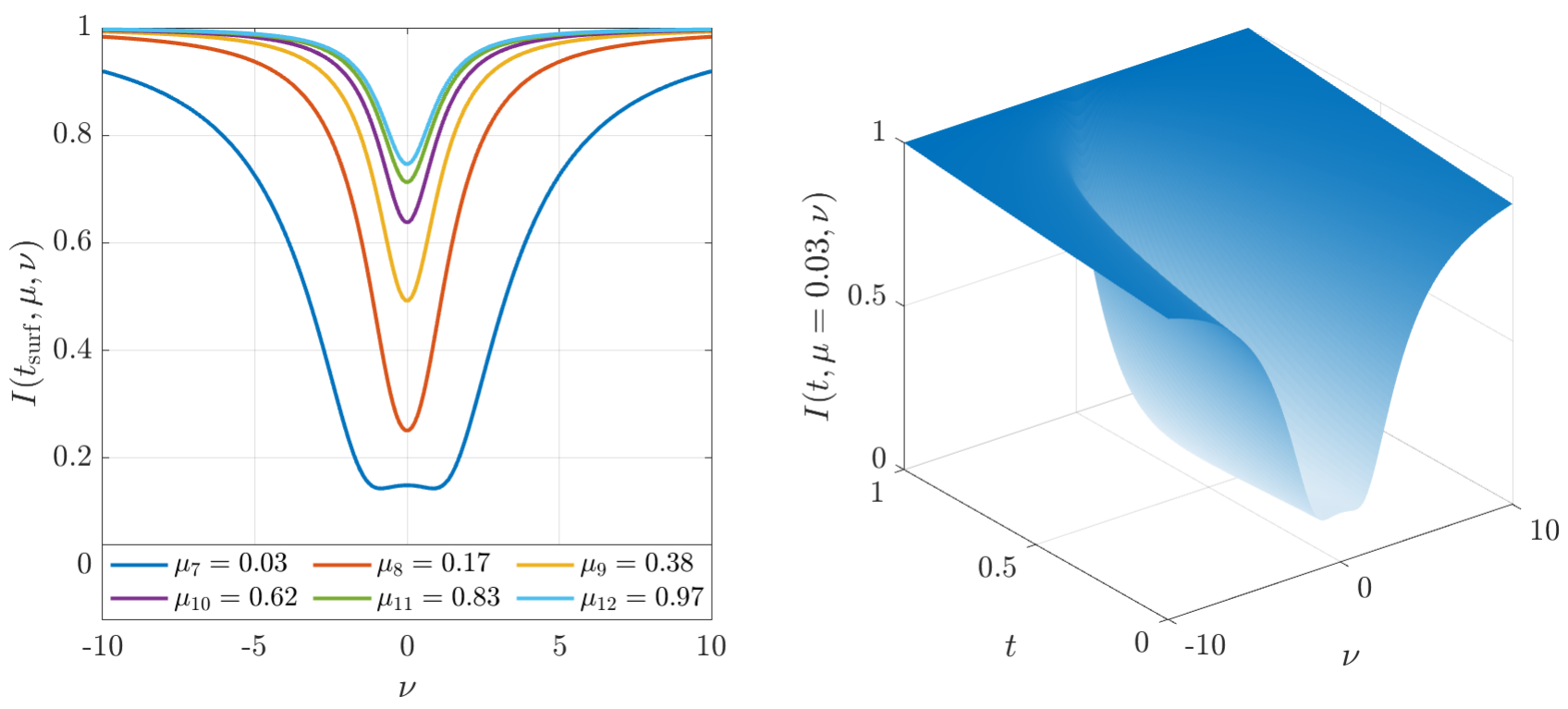}
    \caption{Isotropic CRD solution, obtained with $N_s=N_{\nu}=200$ and $N_\Omega=12$. Left: emerging radiation profiles at $t_{\mathrm{surf}}=\tau_{\mathrm{surf}}=0$ for various directions $\mu$. Right: Solution profile for $\mu=0.0.03$ in the full spatial domain.}
    \label{fig:CRD_sol}
\end{figure}
In the isotropic case, the spectrum inspection also reveals a strong clustering of eigenvalues around one. In Table~\ref{tab:iso}, we show the variation in the number of eigenvalues with absolute value in the interval [0.999,1] w.r.t. the discretization parameters in the case of coherent scattering. In the CRD case, we find an even stronger clustering, with all the corresponding entries exceeding $99.2\%$. In both cases, as expected by the symbol theory, the eigenvalues accumulate at one under refinement. We remark that non-zero (even if relatively small) imaginary parts are present in the coherent case.
We further note that $|\lambda|_{\min}>0.70$ in all cases, signaling a lower bound in the spectrum and the absence of outliers possibly close to zero.  Again, these spectral properties correspond to fast and robust convergence of Krylov methods, as shown in Fig.~\ref{fig:iso_conv} for the coherent case. The GMRES solver, for example, reaches machine precision in 3 more iterations (from 8 to 11), when we refine homogeneously in all dimensions, increasing the total number of degrees of freedom $N$ by a factor of 125.
\begin{figure}[]
  \centering
  \begin{minipage}[c]{0.4\textwidth}
    \centering
    \scalebox{0.9}{
    \begin{tabular}{cc|ccc}
%\toprule
& & \multicolumn{3}{c}{$N_\nu$} \\
& & 10 & 20 & 50 \\
\midrule
\multirow{6}{*}{$N_s$}
& 10  & 98.3\% & 97.9\% & 97.9\% \\
& 20  & 98.4\% & 98.4\% & 98.3\%  \\
& 50  & 98.6\% & 98.7\% & 98.7\%  \\
& 100 & 98.9\% & 99.1\% & 99.2\%
%\bottomrule
\end{tabular}
}
    \captionof{table}{Percentage of eigenvalues with absolute value in the interval [0.999,1] using $N_\mu=12$.}
    \label{tab:iso}
  \end{minipage}
  \hfill
  \begin{minipage}[c]{0.55\textwidth}
    \centering
    \includegraphics[width=\textwidth]{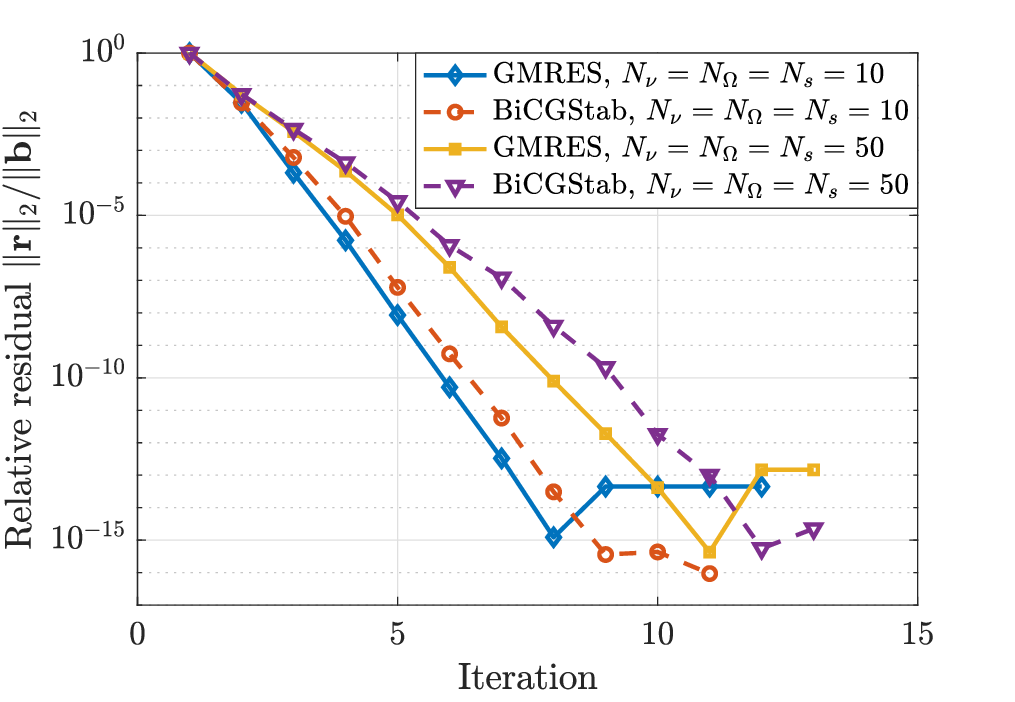}
    \caption{GMRES and BiCGStab convergence for different discretization parameters (from $N=10^3$ to $N=50^3$ degrees of freedom) in case of coherent scattering.}
    \label{fig:iso_conv}
  \end{minipage}%
\end{figure}

\subsection{Realistic 3D problem}\label{sec:3d}
We now consider a realistic RT problem, stemming from an astrophysical application: the synthesis of spectral lines in stellar atmospheres. 
We remark that the spectroscopic observation of remote objects, such as the Sun, is the backbone of remote sensing in astrophysical research. 
In particular, the simulation of line-formation processes, via the solution of a RT problem, is crucial in inferring the physical properties of far-away objects (given the impossibility to directly probe them).
Stellar atmospheres are usually highly inhomogeneous and anisotropic: to accurately model how their properties are encoded in the emitted radiation, it is necessary to consider three-dimensional models that resolve their smallest structures. The resulting RT problems can have a very large number of degrees of freedom (easily exceeding one billion), cf. Figure~\ref{fig:sun}, making high-resolution stellar RT simulations extremely challenging. 

For our final benchmark, we synthesize the resonance line of neutral calcium at 4227\,{\AA} in a 3D model extracted from the R-MHD simulation of \cite{carlsson2016publicly}, which includes the solar atmosphere, from the upper part of the convection zone to the corona. 
%, which captures the structure of the solar photosphere and chromosphere. 
We refer to \cite{benedusi2026rt3d} for the detailed description of the considered model and the boundary conditions. 
The domain has a spatial extension of $6\times6\times2.6$ Mm$^3$ and is discretized with a structured grid with $N_x\times N_y \times N_z$ grid points. The angular grid is a tensor product in azimuth (equidistant nodes) and inclinations (Gauss-Legendre nodes) with $N_\Omega$ total directions and $N_\nu$ frequencies in the interval $F=[4221.6,4232.7]$~{\AA}.
%We use a complete frequency redistribution scattering kernel, considering a two-level atomic model, as described in \cred{CITE}.
We consider a two-level atomic model for neutral calcium, and a dipole scattering kernel describing complete frequency redistribution, as described in \cite{benedusi2026rt3d}.
The atmospheric model provides multiple physical quantities (temperature, density, magnetic field, plasma velocity, etc) which enter the expression of the scattering kernel $\Phi$ and in the physical coefficients $\chi$ and $\sigma$ in \eqref{eq:RT}-\eqref{eq:emissivity}; since those expressions are rather convoluted and not relevant for this work, we refer to \cite{riva2023assessment} for such expressions. 
In the 3D polychromatic setting, $N$ can easily exceed one billion (with dense coupling induced by the scattering integral) and the computation of $A_N$ eigenvalues becomes unfeasible; for this reason, we rely on the robustness of GMRES to evaluate spectral clustering.
The solution is obtained with TRIP\footnote{\url{https://github.com/pietrobe/TRIP}}, which uses a distributed matrix-free approach. The computation of transfer, i.e. the evaluation of $\Lambda$, is obtained with a tailored quadratic exponential integrator, using both short- and long-characteristics, depending on the ray inclination \cite{benedusi2023scalable} and a highly accurate quadrature for scattering \cite{riva2023assessment}. TRIP uses the PETSc GMRES implementation, for which we set a relative tolerance of $10^{-6}$ and 30 iterations as restart.
We run numerical experiments with the MareNostrum 5 supercomputer\footnote{\url{https://www.bsc.es/marenostrum/marenostrum-5}}, using up to 12288 cores of the General Purpose Partition (GPP).

 In table~\ref{tab:stellar}, we report the GMRES iterations needed to converge, varying the discretization parameters along different dimensions, with a problem size varying from $N=3.6\cdot10^7$ to $N=6.5\cdot10^9$. Again, the convergence is robust with respect to all the discretization parameters, as expected from the spectral results.

\begin{figure}
\centering
\begin{tikzpicture}
  % image node
  \node[anchor=south west, inner sep=0] (img) 
    {\includegraphics[width=\linewidth]{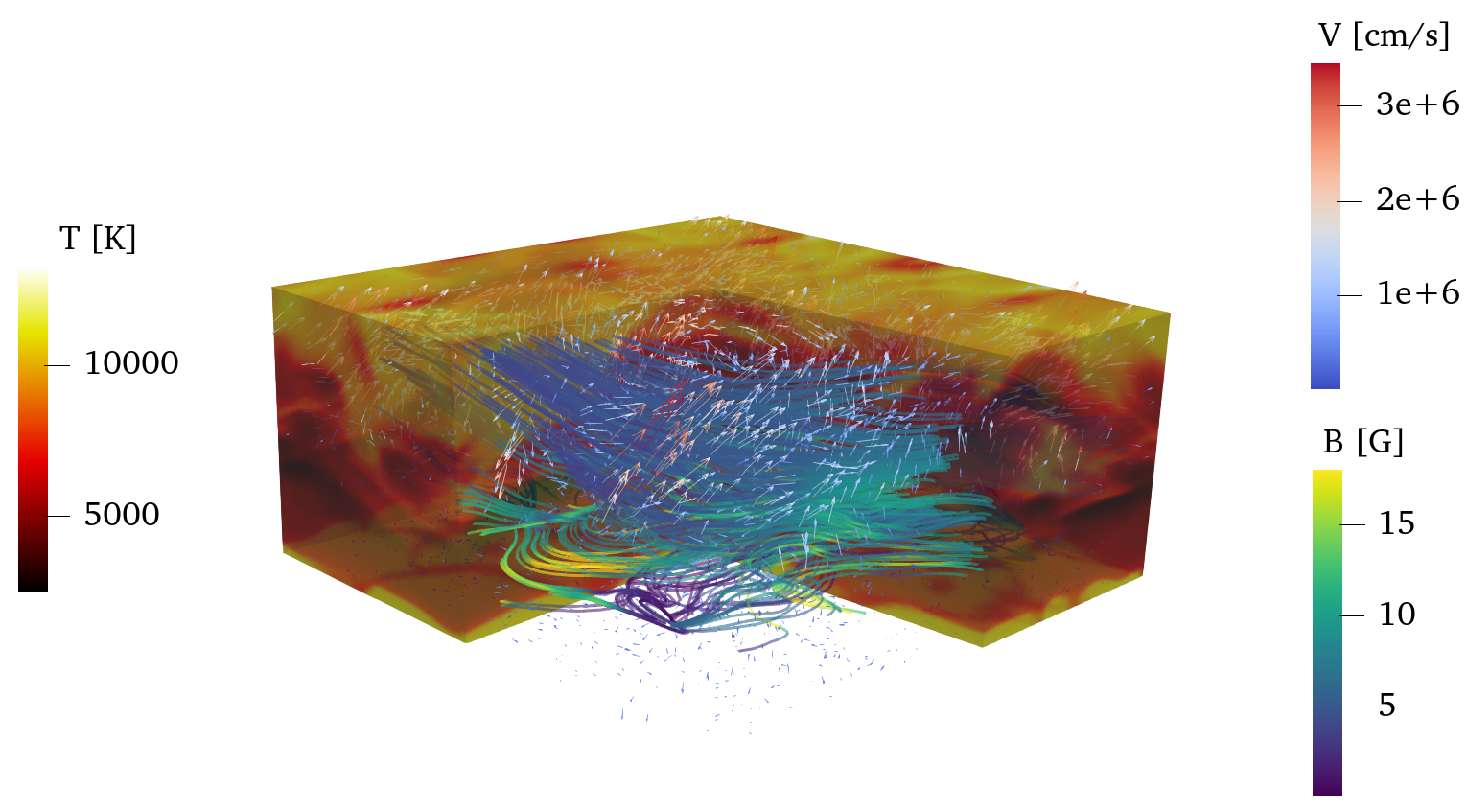}};

  % coordinate system relative to the image
  \begin{scope}[x={(img.south east)}, y={(img.north west)}]

    % origin
    \coordinate (O) at (0.1,0.1);

    % axes
    \draw[->, thick] (O) -- ++(0.04,-0.06) node[right] {$x$};
    \draw[->, thick] (O) -- ++(0.05,0.05) node[above right] {$y$};
    \draw[->, thick] (O) -- ++(0,0.12) node[above] {$z$};

  \end{scope}
\end{tikzpicture}
\caption{3D solar atmosphere model, showing a temperature scalar field and the magnetic and plasma velocities vector fields (streamlines and arrows respectively).}
\label{fig:sun}
\end{figure}

\begin{table}[]
    \centering
    \begin{tabular}{c|c|c|c||c|c|c||r}
         $N_x=N_y$ & $N_z$ & $N_\Omega$ & $N_\nu$ &  $N_s$ & $N_r$ & $N$ & Iterations\\
         \hline
         \hline
          16 & 34  & 64 & 64  & 8704 & 4096  & $3.6\cdot10^7$ & 57\\
          16 & 34  & 64 & 96  & 8704 & 6144  & $5.3\cdot10^7$ & 55 \\
          16 & 34  & 128 & 64 & 8704 & 8192  & $7.1\cdot10^7$ & 57\\
          16 & 34  & 128 & 96 & 8704 & 12288 & $1.1\cdot10^8$ & 55\\
         \hline
          32 & 68 & 64 & 64  & 69632 & 4096  & $2.3\cdot10^8$ & 56\\
          32 & 68 & 64 & 96  & 69632 & 6144  & $4.3\cdot10^8$ & 54\\
          32 & 68 & 128 & 64 & 69632 & 8192  & $5.7\cdot10^8$ & 57\\
          32 & 68 & 128 & 96 & 69632 & 12288 & $8.6\cdot10^8$ & 54\\
         \hline
         63 & 134 & 64 & 64  &  531846 & 4096  & $2.2\cdot10^9$ & 56\\
         63 & 134 & 64 & 96  &  531846 & 6144  & $3.3\cdot10^9$ & 56\\
         63 & 134 & 128 & 64 &  531846 & 8192  & $4.4\cdot10^9$ & 56\\
         63 & 134 & 128 & 96 &  531846 & 12288 & $6.5\cdot10^9$ & 54\\
    \end{tabular}
    \caption{Number of GMRES iterations to convergence for the 3D problem for different degrees of freedom, varying in each dimensions. We recall that $N_s=N_x\cdot N_y\cdot N_z$ is the number of spatial DoFs, $N_r=N_\Omega\cdot N_\nu$ is the number of rays, and $N=N_s\cdot N_r$ is the total number of degrees of freedom.}
    \label{tab:stellar}
\end{table}

\section{Conclusions}\label{sec:concl}
In the current study, we considered a general discretized linear radiative transfer problem and reported the corresponding matrix structures. We proposed a theoretical analysis of the related matrix sequences, explaining the formal reasons why the use of Krylov methods is effective and robust, despite the possibly dense nature of the involved matrices. In short, the compactness of the continuous operators used in the modeling leads to zero-clustered dense matrix sequences plus identity, so that it was possible to formally deduce the clustering at the unity in the spectra of radiative transfer operators under mild assumptions.
As expected from theory, unpreconditioned Krylov methods exhibit robustness with respect to all discretization parameters; in this sense, this setting effectively provides a ``free lunch''. We remark that preconditioning can still be effective to reduce iteration count (e.g. physics-based strategies \cite{janett2024numerical}), even if robustness is already ensured.

We reported a wide set of numerical experiments of increasing modeling complexity, which supports the theoretical result.
Although we use specific discretization settings, we expect the theoretical results and the numerical behavior to hold more broadly because of the underlying compactness of the involved continuous operators. Robustness with respect to the physical parameters can be expected, but remains possibly subject of future investigation, as for the impact of the problem geometry.
As a final note, we did not consider the polarization of radiation in this work, in order to limit notational and theoretical complexity. Nevertheless, the inclusion of polarization would not alter the mathematical structure of the resulting discrete problem, which would simply become vector-valued. Indeed, numerical experiments indicate that robustness \cite{benedusi2021numerical,janett2021numerical} and optimal scaling \cite{benedusi2023scalable} are also present in the polarized case.

This work demonstrates that radiative transfer (RT) simulations can achieve optimal scaling with respect to discretization parameters. This makes high-resolution configurations feasible; such large-scale simulations (with tens of billions of degrees of freedom, see i.e. \cite{benedusi2026rt3d}) are essential for accurate remote sensing applications, for example in the astrophysical context.

\subsection{Acknowledgment}
This work was financed by the Swiss National Science Foundation (SNSF) through the grant 200021-231308. The authors thankfully acknowledge the resources provided by the Barcelona Supercomputing Center through project no. AECT-2024-2-0018, and the help of Tanaus\'u del Pino Alem\'an in processing the 3D atmospheric model.

% \bibliographystyle{siamplain}
% \bibliography{references}

\end{document}